\documentclass{amsart}

\usepackage{amsmath}
\usepackage{amsthm}
\usepackage{amsfonts}
\usepackage{graphicx}
\usepackage{hyperref}
\usepackage{latexsym}
\usepackage{amssymb}
\usepackage{amscd}
\usepackage{booktabs}
\usepackage[small]{caption}
\usepackage[title]{appendix}

\newcommand{\field}[1]{\mathbb{#1}}
\newcommand{\C}{\field{C}}

\newcommand{\dS}{\field{S}}

\newcommand{\cA}{{\mathcal A}}
\newcommand{\cB}{{\mathcal B}}
\newcommand{\cC}{{\mathcal C}}

\newcommand{\cG}{{\mathcal G}}

\newcommand{\cP}{{\mathcal P}}

\newcommand{\sbinom}[2]{\left[ \begin{array}{c} #1 \\ #2 \end{array} \right] }
\newcommand{\sbinomq}[2]{\sbinom{#1}{#2}_q }

\newcommand{\sdist}{\mathrm{d}_{\mathrm{s}}}
\newcommand{\gaussm}[3]{\genfrac{[}{]}{0pt}{}{#1}{#2}_{#3}}
\newcommand{\spparam}{t-(n,k,\lambda)_q}

\title{Subspace Packings -- Constructions and Bounds}

\author[T.\,Etzion]{Tuvi Etzion}
\address{
Tuvi Etzion\\
Technion, Haifa, 
Israel}
\email{etzion@cs.technion.ac.il}

\author[S.\,Kurz]{Sascha Kurz}
\address{
Sascha Kurz\\
University of Bayreuth\\ 
Bayreuth, Germany}
\email{sascha.kurz@uni-bayreuth.de}

\author[K.\,Otal]{Kamil Otal}
\address{
Kamil Otal\\
T\"{U}BITAK BILGEM UEKAE\\
Gebze Turkey} 
\email{kamil.otal@gmail.com}

\author[F.\,\"Ozbudak]{Ferruh \"Ozbudak\vspace*{7ex}}
\address{
Ferruh \"Ozbudak\\
Middle East Technical University\\
Ankara, Turkey}
\email{ozbudak@metu.edu.tr}

\newtheorem{theorem}{Theorem}
\newtheorem{lemma}[theorem]{Lemma}

\newtheorem{remark}[theorem]{Remark}
\newtheorem{construction}[theorem]{Construction}
\newtheorem{claim}[theorem]{Claim}
\newtheorem{corollary}[theorem]{Corollary}
\newtheorem{proposition}[theorem]{Proposition}
\newtheorem{definition}[theorem]{Definition}

\renewcommand{\leq}{\leqslant}
\renewcommand{\geq}{\geqslant}

\newcommand{\F}{\mathbb{F}}

\makeatletter
\DeclareRobustCommand{\sbinom}{\genfrac[]\z@{}}
\makeatother

\newcommand{\deff}{\mbox{$\stackrel{\rm def}{=}$}}

\begin{document}

\begin{abstract}
The {Grassmannian} $\cG_q(n,k)$
is the set of all $k$-dimensional subspaces of the vector space~$\F_q^n$.
It is well known that codes in the Grassmannian
space can be used for error-correction in random network coding.
On the other hand, these codes are $q$-analogs of codes in the Johnson scheme, i.e.
constant dimension codes. These codes of the Grassmannian $\cG_q(n,k)$
also form a family of $q$-analogs of block designs and they are
called \emph{subspace designs}. The application of subspace codes has motivated
extensive work on the $q$-analogs of block designs.

In this paper, we examine one of the last families of $q$-analogs of block designs
which was not considered before. This family called \emph{subspace packings}
is the $q$-analog of packings. This family of designs was considered recently
for network coding solution for a family of multicast networks called the generalized
combination networks. A \emph{subspace packing} $t$-$(n,k,\lambda)^m_q$ is a set $\dS$
of $k$-subspaces from $\cG_q(n,k)$ such that each $t$-subspace of $\cG_q(n,t)$ is contained
in at most $\lambda$ elements of $\dS$. The goal of this work is to consider
the largest size of such subspace packings.


\end{abstract}

\maketitle

\thispagestyle{empty}

\section{Introduction}
\vspace{.5ex} \label{sec1}

Network coding has been attracting increasing attention in the last fifteen years.
The seminal work of Ahlswede, Cai, Li, and Yeung~\cite{ACLY00}
and Li, Yeung, and Cai~\cite{LYC03} introduced the basic concepts of network coding and
how network coding outperforms the well-known routing. This research area was developed
rapidly in the last fifteen years and has a significant influence on other research
areas as well. Random network coding which was introduced in~\cite{HKMKE03,HMKKESL06}
was an important step in the evolution of the research in network coding. One of the direction
which was in the first line of research following the introduction of random network coding was the
design of error-correcting codes for random network coding.
K\"{o}tter and Kschischang~\cite{KoKs08} introduced a framework for error-correction in
random network coding. Their model for the problem introduced a new type of error-correcting codes,
so-called \emph{constant-dimension codes} in the \emph{projective space}. These are sets of $k$-dimensional
subspaces of a finite vector space over a finite field, $k$-subspaces for short, such that each $t$-subspace is
contained in at most one codeword. Defining the subspace distance as $\sdist(U,W)=\dim(U+W)-\dim(U)-\dim(W)
=\dim(U)+\dim(W)-2\dim(U\cap W)$, we can also speak of constant-dimension codes with minimum subspace distance
at least $2k-2t+2$.  Such codes were considered before only in sporadic cases, but their related combinatorial
structures, known as block designs over finite fields were considered throughout the years. They were considered
for their own interest, but also as what is called the $q$-analogs of designs.

The classical theory of \emph{q-analogs} of mathematical objects
and functions has its beginnings in the work of Euler~\cite{Euler,KvA09}.
In 1957, Tits~\cite{Tits57} further suggested that combinatorics of sets could
be regarded as the limiting case $q \to 1$ of combinatorics of vector spaces
over the finite field $\F_q$. Indeed, there is a strong analogy between subsets
of a set and subspaces of a vector space, expounded by numerous
authors---see~\cite{Cohn,GR,Wang} and references therein.
It is therefore natural~to~ask which combinatorial structures
can be generalized from sets (the $q \to 1$ case) to vector spaces
over $\F_q$. For $t$-designs, this question
was first studied by Cameron~\cite{Cam74,Cam74a} and
Delsarte~\cite{Del76} in the early 1970s.
Specifically,~let $\F_q^n$ be a vector space of dimension $n$ over the
finite field $\F_q$.
Then a \emph{$t$-$(n,k,\lambda)$~design over $\F_q$} is defined
in~\cite{Cam74,Cam74a,Del76} as
a collection of $k$-subspaces of $\F_q^n$,
called \emph{blocks}, such that each $t$-subspace of $\F_q^n$
is contained in exactly $\lambda$ blocks. Such $t$-designs over $\F_q$
are the $q$-analogs of conventional combinatorial designs. By analogy
with the $q \to 1$ case, a~$t$-$(n,k,1)$ design over $\F_q$~is said to
be a \emph{$q$-Steiner system}, and is denoted by $\dS_q(t,k,n)$.
$t$-designs over $\F_q$ are often called \emph{subspace designs}. Research in this
area was developed before the introduction of network coding, e.g.~\cite{BKL,MMY95,RaSi89,Suz90,Suz90a,Suz92,Tho87,Tho96}.
But, since the introduction of applications in error-correction for
random network coding by K\"{o}tter and Kschischang~\cite{KoKs08} the research had doubled itself
every year, e.g~\cite{braun2018q,EtSt16} and references therein.

Various $q$-analogs of designs were considered, $t$-designs (see~\cite{braun2018q} and references therein),
Steiner systems~\cite{BEOVW,Etz18}
and in particular the Fano plane~\cite{EtHo18,fano_aut_le_2}, 
transversal designs~\cite{EtSi13}, group divisible designs~\cite{BKKNW},
large sets~\cite{BKKL,BKOW}, etc. But, one very natural modification of the design property was not thoroughly
studied -- the family of packings. A $t-(n,k,\lambda)$ packing is a collection of $k$-subsets (called \emph{blocks}) of
some $v$-set such that every $t$-subset occurs in at most $\lambda$ blocks. Those packings of sets (or vectors in
coding theory language) were extensively studied, see e.g.\ the two surveys~\cite{MiMu92,SWY06}.

A \emph{subspace packing} $\spparam$ is a collection $\cC$ 
of $k$-subspaces (called \emph{blocks} or \emph{codewords}) of $\F_q^n$ such that each $t$-subspace of $\F_q^m$ is contained in at most
$\lambda$ blocks. By $\cA_q(n,k,t;\lambda)$ we denote the maximum number of $k$-subspaces in a $\spparam$ subspace packing
without repeated blocks and by $\cA_q^r(n,k,t;\lambda)$ the corresponding number if repeated blocks are allowed. We have
$\cA_q(n,k,t;\lambda)<\cA_q^r(n,k,t;\lambda)$ if $\lambda$ is large enough. Slightly
abusing notation we write $\cA_1(n,k,t;\lambda)$ and $\cA_1^r(n,k,t;\lambda)$ for the corresponding maximum numbers in the
set case. The special case $\lambda=1$, where we cannot have repeated blocks, corresponds to constant-weight codes. More
precisely, $\cA_1(n,k,t;1)$ is the maximum size of a constant-weight code with length $n$, weight $k$, and minimum Hamming
distance $2k-2s+2$. The corresponding $q$-analog are the constant-dimension codes, mentioned at the beginning of this introduction,
with maximum size $\cA_q(n,k,t;1)$.

The definition of a subspace packing is a straightforward definition for the $q$-analog of a packing for sets. Moreover, subspace packings
have found recently another nice application in network coding. It was proved in~\cite{EtZh18} that the code formed from the dual
subspaces (of dimension $n-k$) of a subspace packing is exactly what is required for a scalar
solution for a family of networks called the \emph{generalized combination networks}. This family of networks
was used in~\cite{EtWa15,EtWa18} to show that vector network coding outperforms scalar linear network coding on multicast networks.
The interested reader is invited to look in these papers for the required definitions and the proofs of the mentioned results.
In~\cite{EtZh18} the authors mainly considered the related network coding problems and
a general analysis of the quantity $\cA_q(n,k,t;\lambda)$. The dual subspaces and the related codes
were also considered in~\cite{EtZh18}. The related quantity $\cB_q(n,k,\delta;\alpha)$ is the maximum number
of $k$-subspaces from $\cG_q(n,k)$ such that each subset of $\alpha$ such $k$-subspaces
span a subspace of $\F_q^n$ whose dimension is $k+ \delta$.

The goal of the current work is to present a study of constructions and upper bounds for the sizes of subspace packings.
Although there are some upper bounds on $\cA_q(n,k,t;\lambda)$ and analysis of subspace packings in~\cite{EtZh18} the topic was
hardly considered in the literature so far. The proceedings paper \cite{EtKuOtOz19} is actually the predecessor of this more extended paper. As mentioned,
for the set case $q=1$ there is a lot of literature. For the other special case $\lambda=1$ and $q>1$ we refer to the
online tables at \url{subspacecodes.uni-bayreuth.de} and the corresponding technical report \cite{heinlein2016tables}.

The rest of this paper is organized as follows.
In Section~\ref{sec_preliminaries} we present basic definitions and some trivial constructions. Various upper bounds for
$\cA_q(n,k,t;\lambda)$ are considered in Section~\ref{sec:upper}. The classic bounds which were obtained in~\cite{EtZh18} will
be revisited as well as other generalizations of the bounds for $\lambda =1$ and also some new upper bounds.
In Section~\ref{sec:constructions} some more constructions to obtain lower bounds on $\cA_q(n,k,t;\lambda)$ will be considered.
In particular, a generalization of what known as the linkage construction will be developed in Section~\ref{subsec_linkage}.
Some special parameters and cases which are not relevant for $\lambda =1$ will be discussed. In Section~\ref{subsec:exact}
the lower and upper bound will be combined to obtain parameters for which the exact value of $\cA_q(n,k,t;\lambda)$ can be
given. Section~\ref{sec:problems} will be devoted for a short conclusion and to identify the main problems for future research.
In Appendix~\ref{sec_tables} we tabulate the best known lower and upper bounds on $\cA_q(n,k,t;\lambda)$ for some small parameters.

\section{Basic Definitions and Constructions}
\label{sec_preliminaries}

For two vectors $u,v \in \F_q^n$ the \emph{Hamming distance} $d_H(u,v)$ is the number of coordinates in which $u$ and $v$ differ.
The \emph{weight} $\operatorname{wt}(v)$ of a vector $v \in \F_q^n$ is the number of nonzero coordinates
in $v$. The \emph{support} of $v$, $\operatorname{supp}(v)$, is the set of nonzero coordinates in $v$,
i.e., $\operatorname{supp}(v) = \{ i : v_i \neq 0\}$.

For two $m \times \eta$ matrices $A$ and $B$ over $\F_q$ the \emph{rank distance} is defined by
$$
d_R (A,B) \deff \text{rank}(A-B)~.
$$
A code $C$ is an $[m \times \eta,\varrho,\delta]$ rank-metric
code if its codewords are $m \times \eta$ matrices over $\F_q$,
they form a linear subspace of dimension $\varrho$ of $\F_q^{m
	\times \eta}$, and for each two distinct codewords $A$ and $B$ we
have that $d_R (A,B) \geq \delta$. Rank-metric codes were well
studied~\cite{Del78,Gab85,Rot91}. It was proved (see~\cite{Rot91})
that for an $[m \times \eta,\varrho,\delta]$ rank-metric code
$\cC$ we have $\varrho \leq
\text{min}\{m(\eta-\delta+1),\eta(m-\delta+1)\}$. This bound is
attained for all possible parameters and the codes which attain it
are called {\it maximum rank distance} codes (or MRD codes in
short).

%

The \emph{Grassmannian} $\cG_q(n,k)$ is the set of all $k$-dimensional subspaces of the vector space~$\F_q^n$.
By $\sbinomq{n}{k}$ we denote its cardinality. We will often consider collections (or multisets) $\cC$ of $k$-dimensional
subspaces in $\F_q^n$. Taking multiplicities into account, their number is denoted by $\#\cC$ or $|\cC|$. Technically,
we might represent such a multiset by a characteristic function $\cC_\chi\colon \cG_q(n,k)\to\mathbb{N}$, where $\cC_\chi(U)$
is the number of times $U\in \cG_q(n,k)$ is contained in $\cC$. With that, we can formally define
$\#\cC=\sum_{U\in \cG_q(n,k)} \cC_\chi(U)$. In the following we will just use the intuitive notions $\#\cC$ and $|\cC|$
without referring to the underlying characteristic function.

A useful counting lemma for chains of subspaces in the Grassmannian is given by:
\begin{lemma}
	\label{lemma_chain_of_subspaces}
	Let $J\le F\le \F_q^n$ be two subspaces of dimensions $j$ and $f$, respectively.
	The number of $u$-subspaces $U$ with $U\cap F=J$ is $q^{(f-j)(u-j)}\gaussm{n-f}{u-j}{q}$.
\end{lemma}

It should be noted that many of the results that are mentioned in this paper were proved
in the context of projective geometry. There is a difference of one in the dimension between
the definitions of vector spaces and the definitions of projective geometry. Throughout the paper
we are using only the notations and the definitions of vector spaces. Hence, if one wants to translate the results into projective geometry, then he should reduce one from all mentioned dimensions.
However, as an abbreviation and by abuse of definitions we find it useful to call $1$-subspaces, $2$-subspaces, $3$-subspaces,
$4$-subspaces, and $(n-1)$-subspaces of an $n$-dimensional vector space
by the names point, lines, planes, solids, and hyperplanes, respectively.

The trivial relations between $\cA_q(n,k,t;\lambda)$ and $\cA_q^r(n,k,t;\lambda)$ are given by
$$
\cA_q(n,k,t;\lambda)\le \cA_q^r(n,k,t;\lambda)\quad\text{and}\quad \cA_q^r(n,k,t;\lambda)\ge \lambda\cdot\cA_q(n,k,t;1)
$$
so that we will mainly study bounds for $\cA_q(n,k,t;\lambda)$. There are a few easy constructions, which we will
list subsequently.

\begin{lemma}
	\label{lemma_all_subspaces}
	For $n,k,t,\lambda\in\mathbb{N}$ with $1\le t\le k\le n$ and $\lambda\ge\sbinomq{n-t}{k-t}$, we have $A_q(n,k,t;\lambda)=\sbinomq{n}{k}$.
\end{lemma}
\begin{proof}
	Take all $k$-subspaces of $\F_q^n$. Each $t$-subspace is contained in exactly $\sbinomq{n-t}{k-t}$ $k$-subspaces.
\end{proof}

Instead of taking all subspaces, we can also take all subspaces that have a certain geometric property.
For example, we can take all $(n-1)$-subspaces not containing a point $P$ and obtain.
\begin{lemma}
	\label{lemma_disjoint_from_point}
	$\cA_q(n,n-1,n-2;q)\ge q^{n-1}$ for $n\ge 3$.
\end{lemma}

Generalizing the idea of Lemma~\ref{lemma_disjoint_from_point} we get:
\begin{lemma}
	\label{lemma_construction_disjoint_a_space}
	For integers $1\le t\le k< n$ we have $\cA_q(n,k,t;q^{(n-k)(k-t)})\ge q^{(n-k)k}$.
\end{lemma}
\begin{proof}
	Take all $k$-subspaces disjoint to a fix $(n-k)$-subspace $F$. We apply Lemma~\ref{lemma_chain_of_subspaces} with
	$f=n-k$, $j=0$, and $u=k$ to deduce that their number is $q^{(n-k)k}$. Similarly, there are $q^{(n-k)t}\cdot\sbinomq{k}{t}$ $t$-subspaces disjoint to $F$.
	As each $k$-subspace contains $\sbinomq{k}{t}$ $t$-subspaces and each $t$-subspace disjoint from $F$ is contained in the same number of $k$-subspaces,
	which are disjoint from $F$, the result follows.
\end{proof}
Applying Lemma~\ref{lemma_construction_disjoint_a_space} with $k=n-a$ and $t=n-2a$ yields the following result.
\begin{corollary}
\label{cor_construction_disjoint_a_space}
For each integers $a\ge 1$ and $n\ge 2a+1$ we have $\cA_q(n,n-a,n-2a;q^{a^2})\ge q^{a(n-a)}$.
\end{corollary}

We can also control the number of covered $t$-subspaces by taking not too many $k$-subspaces.
For example, take arbitrary $\lambda$ out of the $\sbinomq{n}{k}$ $k$-subspaces to obtain the following result.
\begin{lemma}
\label{lemma_trivial}
For integers $1\le t\le k\le n$ and $1\le \lambda < \sbinomq{n-t}{k-t}$ we have $\cA_q(n,k,t;\lambda)\ge \lambda$.
\end{lemma}

\vspace{2ex}
\section{Upper Bounds on the Size of Subspace Packings}
\vspace{-.25ex}
\label{sec:upper}

The ultimate goal when providing an upper bound on the size of a packing is that
it coincides with the lower bound on the size which is obtained by a suitable construction.
Unfortunately, this target is, even for constant-dimension codes, i.e., $\lambda=1$, usually
unattainable. There are various construction methods and lower bounds that are usually improved
with the time. But, except for some basic upper bounds, there are only a handful of methods to
improve them and usually the improvements are not dramatic.

Obviously, we have $\cA_q(n,k,t;\lambda)\le \cA^r_q(n,k,t;\lambda)$ and $\cA_q(n,k,t;\lambda)\le \sbinomq{n}{k}$. For
$\lambda=1$ no repeated blocks can occur, so that $\cA_q(n,k,t;\lambda)= \cA^r_q(n,k,t;\lambda)$.
Arguably, the simplest non-trivial upper bound arises from a packing argument. The ambient space
$\F_q^n$ contains exactly $\sbinomq{n}{t}$
$t$-subspaces and each codeword (a $k$-subspace) contains exactly $\sbinomq{k}{t}$ $t$-subspaces, so that:
\begin{proposition}
\label{prop_packing}
For any positive integers $1\le t\le k\le n$ and $1 \leq \lambda \leq \sbinomq{n}{k}$ we have that
$$\cA_q(n,k,t;\lambda)\le \cA_q^r(n,k,t;\lambda) \leq  \left\lfloor\lambda\frac{\sbinomq{n}{t}}{\sbinomq{k}{t}}\right\rfloor~.$$
\end{proposition}
Proposition~\ref{prop_packing} is well-known as the packing bound.
Equality in Proposition~\ref{prop_packing} is attained only for subspace designs. However, the upper bound can be asymptotically
achieved for fixed parameters $q$, $k$, and $t$, see \cite{blackburn2012asymptotic,frankl1985near} (noting that it suffices to consider
the special case $\lambda=1$). In other words, it is not possible to improve the upper bound of Proposition~\ref{prop_packing} by some
constant factor if the dimension $n$ of the ambient space tends to infinity (while all other parameters are kept fixed). This asymptotic
statement can be made more concrete by comparing the upper bound of Proposition~\ref{prop_packing} with the construction using lifted MRD
codes, see Construction~\ref{const1} in Section~\ref{subsec_linkage} for a description of the lifted MRD codes. In~\cite[Proposition 8]{heinlein2017asymptotic} this was done for $\lambda=1$, so that we directly state the slight
reformulation:
\begin{theorem}
	For $k\le n-k$ we have
	$$
	\lambda q^{t(n-k)} \leq \cA_q^r (n,k,t;\lambda) \leq   \frac{(1/q;1/q)_{k-t}}{(1/q;1/q)_{k}} \cdot \lambda q^{t(n-k)},
	$$
	where $(1/q;1/q)_n=\prod_{i=1}^{n}\left(1-1/q^i\right)$ is the specialized $q$-Pochhammer symbol, see e.g.~\cite{gasper2004basic}
	for some background, and
	$$
	\frac{(1/q;1/q)_{k-t}}{(1/q;1/q)_{k}}\le
	\frac{q-1}{q\cdot (1/q;1/q)_{k}}\le \frac{q-1}{q\cdot (1/q;1/q)_{\infty}}\le \frac{1}{2\cdot (1/2;1/2)_\infty}< 1.7314.
	$$
\end{theorem}
So, even for the binary case $q=2$, no dramatic improvements are possible. Moreover, with increasing field size $q$ the factor
$\frac{q-1}{q\cdot (1/q;1/q)_{\infty}}$ quickly tends to one.

The condition $k\le n-k$ is necessary for the existence of the underlying MRD code. For $\lambda=1$ and positive integers
$1\le t\le k\le n$ we can use duality to obtain
\begin{eqnarray}
&&\cA_q(n,k,t;1)=\cA_q (n,n-k,n-2k+t;1)\quad\text{and}\quad\\ 
&&\cA_q^r(n,k,t;1)=\cA_q^r(n,n-k,n-2k+t;1),
\end{eqnarray}
so that the restriction $k\le n-k$ is irrelevant. For $\lambda>1$ this is different and the cases $k>\frac{n}{2}$ turn out to
be more interesting.

In Subsection~\ref{subsec_q_analog_classical} we will study $q$-analogs of classical upper bounds for packings. Improvements
for $q>1$ based on the theory of $q^r$-divisible codes are the topic of Subsection~\ref{subsec_divisible}. Additional
upper bounds are summarized in Subsection~\ref{subsec_more_upper_bounds}, which mainly targets the cases where $2k>n$ and $\lambda>1$.

\subsection{$q$-analogs of classical bounds}
\label{subsec_q_analog_classical}
Of course the upper bound of Proposition~\ref{prop_packing} is a $q$-analog of a classical bound. Since any $k$-set
contains ${k\choose t}$ subsets of size $t$ and every $t$-set is covered at most $\lambda$ times, we have
$\cA_1^r (n,k,t;\lambda)\le \left\lfloor \lambda {n\choose t}/{k \choose t}\right\rfloor$. For fixed values $k$ and
$t$ this upper bound can be asymptotically attained, see \cite{rodl1985packing}. (Note that it suffices to consider
the case $\lambda=1$, since those examples can be taken $\lambda$-fold.)

As observed by Sch\"onheim \cite{schonheim1966maximal} we have
\begin{equation}
\label{eq_2}
\cA_1^r (n,k,t;\lambda)\le \left\lfloor \frac{n}{k}\cdot \cA_1^r (n-1,k-1,t-1;\lambda)\right\rfloor,
\end{equation}
which directly generalizes to:
\begin{proposition}
	\label{johnson_bound_point}
	If $n$, $k$, $t$, and $\lambda$ are positive integers such that $2 \leq t \le k \le n$ and $\lambda\ge 1$,
	then
	$$
	\cA_q^r (n,k,t;\lambda) \leq \left\lfloor \frac{q^n -1}{q^k-1} \cA_q^r (n-1,k-1,t-1;\lambda)  \right\rfloor
	$$
	and
	$$
	\cA_q (n,k,t;\lambda) \leq \left\lfloor \frac{q^n -1}{q^k-1} \cA_q (n-1,k-1,t-1;\lambda)  \right\rfloor.
	$$
\end{proposition}
\begin{proof}
	Let $\mathcal{C}$ be a subspace packing attaining $\cA_q(n,k,t;\lambda)$ (or $\cA_q^r(n,k,t;\lambda)$). For each point $P$ in $\mathbf{F}_q^n$
	let $\mathcal{C}_P$ be the collection of blocks of $\mathcal{C}$ that contain $P$. Moding $P$ out we see
	$\#\mathcal{C}_P\le \cA_q(n-1,k-1,t-1;\lambda)$
	(or $\#\mathcal{C}_P\le \cA_q^r(n-1,k-1,t-1;\lambda)$). Since
	$\mathbb{F}_q^n$ contains $\gaussm{n}{1}{q}$ points, any block (of $\mathcal{C}$) contains $\gaussm{k}{1}{q}$ points,
	$\gaussm{n}{1}{q}/\gaussm{k}{1}{q}=\frac{q^n-1}{q^k-1}$, and $\#\mathcal{C}$ is an integer, the stated bounds follow.
\end{proof}

For $\lambda=1$ inequality~(\ref{eq_2}) was also obtained by Johnson in \cite{johnson1962new} and reformulated to its $q$-analog, c.f.\
Proposition~\ref{johnson_bound_point}, in \cite[Theorem 3]{xia2009johnson}, see also \cite{EtVa11}. Due to the latter references we also
speak of the Johnson bound. Another proof of Proposition~\ref{johnson_bound_point} can also be found in~\cite{EtZh18}.

An easy implication of Proposition~\ref{johnson_bound_point} is:
\begin{lemma}
	\label{lemma_upper_bound_exclude_point}
	For $n\ge 3$ we have $\cA_q(n,n-1,n-2;q)\le \cA^r_q(n,n-1,n-2;q)\ \le q^{n-1}$.
\end{lemma}
\begin{proof}
	By Proposition~\ref{prop_packing} we have that
	$$\cA^r_q(3,2,1;q)\le \left\lfloor\frac{\left(q^2+q+1\right)\cdot q}{q+1}\right\rfloor=\left\lfloor q^2+\frac{q}{q+1}\right\rfloor=q^2~,$$
	For $n\ge 4$ we inductively apply Proposition~\ref{johnson_bound_point} and obtain
	$$
	\cA^r_q(n,n-1,n-2;q)\le\left\lfloor\frac{\sbinomq{n}{1} \cdot q^{n-2}}{\sbinomq{n-1}{1}}\right\rfloor=
	\left\lfloor q^{n-1}+\frac{q^{n-2}}{\sbinomq{n-1}{1}}\right\rfloor=q^{n-1}.
	$$
\end{proof}

By recursively applying Proposition~\ref{johnson_bound_point}, taking the basis $t=1$ and then applying
$\cA^r_q(n,k,1;\lambda)\le\left\lfloor \lambda \sbinomq{n}{1}/\sbinomq{k}{1}\right\rfloor$ gives a tighter bound than
Proposition~\ref{prop_packing}. More precisely, for $\cA^r_q(n,k,t;\lambda)$ applying Proposition~\ref{johnson_bound_point} $t-1$ times
without rounding down gives
$$
\cA^r_q(n,k,t;\lambda)\le \prod_{i=0}^{t-2} \frac{q^{n-i}-1}{q^{k-i}-1}\cdot \cA^r_q(n-t+1,k-t+1,1;\lambda).
$$
Plugging in $\cA^r_q(n',k',1;\lambda)\le \lambda \sbinomq{n'}{1}/\sbinomq{k'}{1}$ yields
$$
\cA^r_q(n,k,t;\lambda)\le \lambda\cdot \prod_{i=0}^{t-1} \frac{q^{n-i}-1}{q^{k-i}-1}=\lambda\cdot \sbinomq{n}{t}/\sbinomq{k}{t}.
$$
Rounding in the iterations might decrease the bounds, while the relative difference gets negligible for large values of $t$, c.f.\
\cite{heinlein2017asymptotic}.

Instead of blocks containing a certain point $P$, we can also consider the collection of blocks that are contained in a certain
hyperplane $H$.
\begin{proposition}
	\label{johnson_bound_hyperplane}
	If $n$, $k$, $t$, and $\lambda$ are positive integers such that $1 \leq t \leq k \leq n$, then
	$$\cA_q^r(n,k,t;\lambda)\leq \left\lfloor\frac{q^n-1}{q^{n-k}-1}\cdot \cA_q^r(n-1,k,t;\lambda)\right\rfloor$$
	and
	$$\cA_q(n,k,t;\lambda)\leq \left\lfloor\frac{q^n-1}{q^{n-k}-1}\cdot \cA_q(n-1,k,t;\lambda)\right\rfloor~.$$
\end{proposition}
\begin{proof}
	Let $\mathcal{C}$ be a subspace packing attaining $\cA_q(n,k,t;\lambda)$ (or $\cA_q^r(n,k,t;\lambda)$). For each hyperplane $H$ in $\mathbf{F}_q^n$
	let $\mathcal{C}_H$ be the collection of blocks of $\mathcal{C}$ that are contained in $H$. Embedded in the $(n-1)$-dimensional vector space
	$H\simeq \mathbb{F}_q^{n-1}$ we see $\#\mathcal{C}_H\le \cA_q(n-1,k,t;\lambda)$ (or $\#\mathcal{C}_H\le \cA_q^r(n-1,k,t;\lambda)$). Since
	$\mathbb{F}_q^n$ contains $\gaussm{n}{n-1}{q}$ hyperplanes, any block of $\mathcal{C}$ is contained in $\gaussm{n-k}{n-k-1}{q}$ hyperplanes,
	$\gaussm{n}{n-1}{q}/\gaussm{n-k}{n-k-1}{q}=\frac{q^n-1}{q^{n-}k-1}$, and $\#\mathcal{C}$ is an integer, the stated bound follows.
\end{proof}
For $q=1$ this bound is well known, the case $q>1$, $\lambda=1$ is treated in \cite{EtVa11}, and the general case is also proven in \cite{EtZh18}.

The combination of the packing bound in Proposition~\ref{prop_packing} and the Johnson-type bound for $(n-1)$-subspaces
of Proposition~\ref{johnson_bound_hyperplane} gives the following improvement:
\begin{proposition}
	\label{prop_combination_packing_johnson_hyperplane}
	If $n$, $k$, $t$, and $\lambda$ are positive integers such that $1 \leq t< k< n$,
	then
	$$
	\cA_q(n,k,t;\lambda)\le \underset{0\le x\le \cA_q(n-1,k,t;\lambda)}{\max} \!\!\!\!\!\!\!\!\!\min\left\{
	x+\left\lfloor\frac{\lambda\sbinomq{n-1}{t}-x\sbinomq{k}{t}}{\sbinomq{k-1}{t}}\right\rfloor,
	\left\lfloor\frac{q^n-1}{q^{n-k}-1}\cdot x\right\rfloor
	\right\}
	$$
	and
	$$
	\cA^r_q(n,k,t;\lambda)\le \underset{0\le x\le \cA^r_q(n-1,k,t;\lambda)}{\max} \!\!\!\!\!\!\!\!\!\min\left\{
	x+\left\lfloor\frac{\lambda\sbinomq{n-1}{t}-x\sbinomq{k}{t}}{\sbinomq{k-1}{t}}\right\rfloor,
	\left\lfloor\frac{q^n-1}{q^{n-k}-1}\cdot x\right\rfloor
	\right\}.
	$$
\end{proposition}
\begin{proof}
	Let $\cC$ be a subspace packing with matching parameters and $H$ be an arbitrary hyperplane of $\mathbb{F}_q$. By $x$
	we denote the number of blocks of $\cC$ that are contained in $H$ and by $y$ those that are not contained
	in $H$, so that $\#\cC=x+y$. The $x$ blocks contained in $H$ cover $x\sbinomq{k}{t}$ out of the
	$\lambda\sbinomq{n-1}{t}$ $\lambda$-fold $t$-subspaces of $H$. Any of the $y$ codewords not contained in $H$ covers
	exactly $\sbinomq{k-1}{t}$ $t$-subspaces in $H$, so that $y\le \left\lfloor\frac{\lambda\sbinomq{n-1}{t}
		-x\sbinomq{k}{t}}{\sbinomq{k-1}{t}}\right\rfloor$. The largest possible value for $x$, call it $x^\star$,
	clearly gives the tightest such upper bound on $\cC$. Now assume that every hyperplane of $\mathbb{F}_q^n$
	contains at most $x^\star$ codewords, then counting gives $\#\cC\le \left\lfloor\frac{q^n-1}{q^{n-k}-1}\cdot x^\star\right\rfloor$.
\end{proof}
In order to compare the different bounds, consider a numerical example for the parameters $\cA^r_2(5,3,2;2)$.
Proposition~\ref{prop_packing} and Proposition~\ref{johnson_bound_point} give
$\cA^r_2(5,3,2;2)\le 44$, while Proposition~\ref{johnson_bound_hyperplane} gives $\cA^r_2(5,3,2;2)\le 82$.
A bit better, Proposition~\ref{prop_combination_packing_johnson_hyperplane} gives $\cA^r_2(5,3,2;2)\le 41$, where the corresponding maximum is
attained at $x=4$. Later on this bound will be improved. However, Proposition~\ref{prop_combination_packing_johnson_hyperplane} also
gives $\cA^r_2(7,4,3;3)\le 2358$, which is still the best known upper bound. Here the maximum is attained at $x=130$.
Let us consider another example which goes a bit beyond the simple estimation of Proposition~\ref{prop_combination_packing_johnson_hyperplane}.
For $\cA^r_2(7,5,1;3)$ we obtain the upper bound $11$, which is uniquely attained at $x=1$. How would the intersection of such a
subspace packing with a hyperplane containing exactly one block look like? We would have one $5$-subspace and ten $4$-subspaces in $\F_2^6$
such that every point is covered at most triple-fold. Indeed we can show that such a configuration cannot exist,\footnote{Using the methods of
	Subsection~\ref{subsec_divisible}, we can consider the corresponding multiset $\cP$ of points, which has cardinality $181$ and is $2^3$-divisible.
	Its $3$-complement $\overline{\cP}$ is also $8$-divisible and has cardinality $8$, which leaves an $8$-fold point as the unique possibility
	for $\overline{\cP}$. Due to $\lambda=3<8$, this is impossible in our situation. We remark that the Johnson bound for points, see
	Proposition~\ref{johnson_bound_point}, gives $\cA^r_2(7,5,1;3)\le 12$, while its improvement based on the methods of
	Subsection~\ref{subsec_divisible}, see Proposition~\ref{johnson_bound_point_improved}, gives $\cA^r_2(7,5,1;3)\le 11$.} which shows
$\cA_2^r(7,5,1;3)\le 10$. From a higher perspective, this example suggests to study $t-(n,\ge\!\!k,\lambda)_q$ subspace packings, i.e.,
collections of subspaces in $\F_q^n$ of dimension at least $k$ such that each $t$-subspace of $\F_q^n$ is covered at most $\lambda$ times.
Quite naturally, things will get more complicated then. To this end, for the special case $\lambda=1$ a related stream of literature might
be mixed-dimension subspace codes, generalizing constant-dimension codes in the same way, see e.g.~\cite{ubt_eref45940} for a recent survey,
or generalized vector space partitions \cite{generalized_vector_space_partitions}.

When $q=1$, $\lambda=1$, and $n<k^2/(t-1)$ there is another bound also due to Johnson \cite{johnson1962new} which is
often smaller than the previously mentioned Johnson bound. This bound is obtained by letting $m$ denote the number
of codewords and writing $km=nl+r$, where $0\le r<n$. Counting the number of pairs of codewords that both contain a fixed
element and summing over all possible choices gives
$$
nl(l-1)+2lr\le (t-1)m(m-1),
$$
which implies, the slightly weaker variant,
$$
\cA_1(n,k,t;1)\le\left\lfloor \frac{(k+1-t)n}{k^2-(t-1)n}\right\rfloor.
$$
This second Johnson bound was generalized in~\cite[Theorem 2]{xia2009johnson} to $q\ge 2$:
\begin{theorem}
	If $(\left(q^k-1\right)^2>\left(q^n-1\right)\left(q^{t-1}-1\right)$, then
	$$
	\cA_q(n,k,t;1)\le \frac{\left(q^k-q^{t-1}\right)\left(q^n-1\right)}{\left(q^k-1\right)^2-\left(q^n-1\right)\left(q^{t-1}-1\right)}
	$$
\end{theorem}
However, different to the case of constant
weight codes studied by Johnson, the required condition is quite restrictive. In \cite[Proposition 1]{heinlein2017asymptotic} it was
shown that it is only satisfied for $t=1$, where the bound collapses to $\cA_q(n,k,t;1)\le \frac{q^n-1}{q^k-1}$ and indeed tighter upper bounds are available.

%

\subsection{Upper bounds based on $q^r$-divisible codes}
\label{subsec_divisible}
As we have seen in the previous subsection for the example of packings, when we consider the $q$-analog of a classical
combinatorial object often there also exist $q$-analogs of the classical bounds. For designs the known necessary existence
criteria also have their $q$-analog counterparts. Interestingly enough, for group divisible designs there is an additional
necessary existence criterion for $q>1$, see \cite{BKKNW}. Also the Johnson bound for constant-dimension codes, see
Proposition~\ref{johnson_bound_point} for $\lambda=1$, was improved \cite{upper_bounds_cdc}. These improvements are based on
the theory of $q^r$-divisible codes, which we will briefly introduce in this subsection.

A \emph{$q^r$-divisible code} is a linear block code (over $\mathbb{F}_q$) in the Hamming scheme where all weights are divisible
by $q^r$. This family of codes has been introduced by Ward~\cite{ward1999introduction}. The main relation between collections
of subspaces of $\F_q^n$ and $q^r$-divisible codes is:

\begin{lemma}(\cite[Lemma 4]{upper_bounds_cdc})
	\label{lemma_is_divisible}
	Let $\cP$ be the multiset of $1$-subspaces generated by a non-empty multiset of subspaces of $\F_q^n$ all having dimension at least
	$k \ge 2$ and let $H$ be an $(n-1)$-subspace of $\F_q^n$. Then,
	$
	\left|\cP\right| \equiv \left|\cP \cap H\right| \pmod {q^{k-1}}
	$.
\end{lemma}
If we form a generator matrix from the column vectors associated with $\cP$, i.e. one representative from each $1$-subspace,
then the generated code will be a linear $q^{k-1}$-divisible code. Let $c$ be a codeword of the code and $H$ be the corresponding
hyperplane. Then, $\operatorname{wt}(c) =\left|\cP\right| - \left|\cP \cap H\right|$, which is divisible by $q^{k-1}$. So, we also
say that the multiset $\cP$ is $q^{k-1}$-divisible if $\left|\cP\right| \equiv \left|\cP \cap H\right| \pmod {q^{k-1}}$ for every
hyperplane $H$ of $\F_q^n$.

We associate a multiset $\cP$ with a weight function $\omega$ that counts the multiplicity of every point of $\F_q^n$.
If $\lambda$ is an upper bound for $\omega$, we define the $\lambda$-complement $\overline{\cP}$ of $\cP$ via the weight
function $\lambda-\omega(P)$ for every point $P$ in $\F_q^n$. As shown in~\cite[Lemma 2]{upper_bounds_cdc} we also have
$|\overline{\cP}| \equiv |(\overline{\cP}\cap H)| \pmod {q^{k-1}}$ for every hyperplane $H$, i.e., a $q^{k-1}$-divisible
code of length $|\overline{\cP}|$ must exist.

As an example consider the following application of the Johnson bound, see Proposition~\ref{johnson_bound_point}:
$$
\cA_2(9,4,2;1)\le \left\lfloor\sbinomq{9}{1} \cA_2(8,3,1;1)/\sbinomq{4}{1} \right\rfloor=
\left\lfloor\frac{17374}{15}\right\rfloor=\left\lfloor1158+\frac{4}{15}\right\rfloor~.
$$
If $1158$ would be attained, then there would be a $2^3$-divisible code of length $4$. For cardinality
$1157$ there would be a $2^3$-divisible code of length $4+15=19$. Since no such codes exist, we have
$\cA_2(9,4,2;1)\le 1156$. Fortunately, the possible lengths of $q^r$-divisible codes over $\F_q$ have been
completely characterized in \cite{upper_bounds_cdc}. Each $t$-subspace is $q^{t-1}$-divisible such that each
$q^j$-fold copy of an $(t-j)$-subspace is $q^{t-1}$-divisible for all $0\le j<t$. Via concatenation
we see that there exists a $q^r$-divisible code of length $n=\sum_{i=0}^r a_i\cdot q^i\cdot\sbinomq{r+1-i}{1}$
for all $a_i\in\mathbb{N}_{\ge 0}$ for $0\le i\le r$. \cite[Theorem 4]{upper_bounds_cdc} states that a
$q^r$-divisible code of length $n$ exists if and only if $n$ admits such a representation as a non-negative
integer linear combination of $q^i\cdot\sbinomq{r+1-i}{1}$ for $0\le i\le r$. Moreover, if
$n=\sum_{i=0}^r a_i\cdot q^i\cdot\sbinomq{r+1-i}{1}$ with $0\le a_i\le q-1$ for $0\le i\le r-1$ and $a_r<0$,
then no $q^r$-divisible code of length $n$ exists. In our example of $2^3$-divisible codes the possible
summands are $15$, $14$, $12$, and $8$. The representations $4=0\cdot 15+0\cdot 14+1\cdot 12-1\cdot 8$ and
$19=1\cdot 15+0\cdot 14+1\cdot 12-1\cdot 8$ implies that no $2^3$-divisible codes of lengths $4$ or $19$ exists.
We can reduce until the remainder is a possible length of a $q^{k-1}$-divisible code. For this
purpose we define

\begin{definition}
	\label{def:newdef}
	Let $\left\{a / \sbinomq{k}{1}\right\}_k$ denote the maximum $b\in\mathbb{N}$ for which
	$a-b\cdot \sbinomq{k}{1}$ is a non-negative integer that is attained as length of some $q^{k-1}$-divisible code.
\end{definition}

An efficient algorithm for the computation of $\left\{a / \sbinomq{k}{1} \right\}_k$ was given in~\cite{upper_bounds_cdc}.
The Johnson bound is improved as follows.
\begin{proposition}
	\label{johnson_bound_point_improved}
	If $n$, $k$, $t$, and $\lambda$ are positive integers such that $2\le t\le k\le
	n$, 
	then
	$$\cA_q(n,k,t;\lambda)\le \left\{\sbinomq{n}{1} \cdot \cA_q(n-1,k-1,t-1;\lambda)/\sbinomq{k}{1} \right\}_k,$$
	$$\cA^r_q(n,k,t;\lambda)\le \left\{\sbinomq{n}{1} \cdot \cA^r_q(n-1,k-1,t-1;\lambda)/\sbinomq{k}{1} \right\}_k,$$
	and
	$$\cA_q(n,k,1;\lambda)\le \cA^r_q(n,k,1;\lambda)\left\{\sbinomq{n}{1}/\sbinomq{k}{1} \right\}_k.$$
\end{proposition}
\begin{proof}
	Let $\cP$ be the $q^{k-1}$-divisible multiset of points of the subspace packing, see Lemma \ref{lemma_is_divisible}. In $\cP$ every point has
	multiplicity at most $\cA_q(n-1, k-1, t-1; \lambda)$ so that the $\cA_q(n-1, k-1, t-1; \lambda)$-complement is also
	$q^{k-1}$-divisible. Thus, the claim follows from Definition~\ref{def:newdef}. We can use the same argument for the case where repeated blocks
	are allowed.
\end{proof}
Proposition~\ref{johnson_bound_point_improved} gives
$
\cA^r_2(6,4,3;2) \leq  \left\{63\cdot \cA_2(5,3,2;2)/ 15\right\}_4=
\left\{63\cdot 32/ 15\right\}_4$ $=132
$,
while the Johnson bound in Proposition~\ref{johnson_bound_point} only gives $\cA^r_2(6,4,3;2)\le 134$.
This specific bound is further improved in the next subsection, where we focus on the situation for $2k>n$. Another example,
which is indeed tight, is $\cA^r(8,3,1;3)=107$, where the Johnson bound in Proposition~\ref{johnson_bound_point} only gives
$\cA^r(8,3,1;3)\le 109$. The improvement is based on the fact that there is no $2^2$-divisible code of length $n=9$ over $\F_2$.

For $\lambda=1$ there is a very clear picture for the best known upper bounds for $\cA_q(n,k,t;1)$. Due to duality we can assume
$2k\le n$. The recursive bound of Proposition~\ref{johnson_bound_point_improved} refers back to the case of partial spreads, i.e., $t=1$.
All known upper bounds for partial spreads can be concluded from the non-existence of projective divisible codes, see \cite{honold2018partial}
for a survey. So far these bounds are only improved for the two cases $\cA_2(6,3,2;1)=77$ \cite{hkk77} and
$\cA_2(8,4,2;1)=257$ \cite{heinlein2017classifying}, which are both based on exhaustive integer linear programming computations,
c.f.~Section~\ref{sec:IntProg}. So, one might expect that it is hard to find a better general bound than the improved Johnson bound of
Proposition~\ref{johnson_bound_point_improved} for the cases with $2k\le n$. For the more general $t-(n,\ge\!\!k,\lambda)_q$ subspace packings,
mentioned and introduced after the discussion of Proposition~\ref{prop_combination_packing_johnson_hyperplane}, the approach of the
improved Johnson bound also looks promising, c.f.~\cite{JohnsonMDC}, where this technique was applied to mixed-dimension subspace codes.

\subsection{Additional upper bounds}
\label{subsec_more_upper_bounds}
As mentioned at the beginning of Section~\ref{sec:upper}, the cases where $2k>n$ and $\lambda>1$ seem to be somehow different.
So, in this subsection we try to develop tighter upper bounds for the cases when the dimension $k$ of the blocks
is large compared to the dimension $n$ of the ambient space.

Another approach for upper bounds is to invoke the vector space structure of subspaces, i.e., to apply dimension arguments.
\begin{lemma}
	\label{lemma_intersection_upper_bound}
	Let $\lambda,n,k,t$ be positive integers with $1\le t\le k\le n$, $1 \leq \lambda < \sbinomq{n-t}{k-t}$, and
	$(\lambda+1)k-\lambda n\ge t$, then $A^r_q(n,k,t;\lambda)\le \lambda$.
\end{lemma}
\begin{proof}
	Since the intersection $A\cap B$ of an $a$-subspace $A$ and a $b$-subspace $B$ in $\mathbb{F}_q^n$ has a dimension of at least $a+b-n$
	we inductively obtain that the intersection of $\lambda+1$ $k$-subspaces is at least $(\lambda+1)k-\lambda n$.
\end{proof}

If $k>\tfrac{n}{2}$ we have that each two blocks intersect non-trivially, which implies the following
recursive bound.
\begin{proposition}
\label{prop_intersection_recursion}
If $\lambda>1$, $k>\tfrac{n}{2}$, and $t\le 2k-n$, then
$$
\cA^r_q(n,k,t;\lambda)\le 1+\cA^r_q(k,2k-n,t;\lambda-1).
$$
\end{proposition}
\begin{proof}
	Let $\cC$ be an $\spparam$ subspace packing and $C$ be an arbitrary block of $\cC$. For any other block $C'\in\cC$ we have $\dim(C\cap C')\ge 2k-n$. For each block
	$C'\in\cC\backslash\{C\}$ we pick an $(2k-n)$-subspace of $C\cap C'$, so that we obtain an $t-(k,2k-n,\lambda-1)_q$ subspace packing $\cC'$ of cardinality
	$\#\cC-1$.
\end{proof}
We remark that in general we can only directly conclude $\cA_q(n,k,t;\lambda)\le 1+\cA^r_q(k,2k-n,t;\lambda-1)$, since several different intersections
$C\cap C'$ may be mapped to the same $(2k-n)$-subspace in $\cC'$. An illustrating example is $\cA_2(6,4,2;4)\ge 52>1+\cA_2(4,2,2;3)=1+\gaussm{4}{2}{2}=36$.
However, in several cases the best known upper bound for $\cA_q(n,k,t;\lambda)$ is the same as for $\cA^r_q(n,k,t;\lambda)$, so that we can obtain good
results anyway. An example is $\cA_2(8,5,1;2)\le\cA^r_2(8,5,1;2)\le 1+\cA^r_2(5,2,1;1)\le 10$, where the last inequality is obtained from the packing
bound, see Proposition~\ref{prop_packing}. Indeed, $\cA_2(8,5,1;2)=\cA^r_2(8,5,1;2)=10$ can be attained. Similarly, we have $18\le \cA_2(8,5,1;3)\le
\cA^r_2(8,5,1;3)\le 1+\cA^r_2(5,2,1;2)\le 21$ and $27\le \cA_2(8,5,1;4)\le
\cA^r_2(8,5,1;4)\le 1+\cA^r_2(5,2,1;3)\le 31$, where also integer linear programming does not give better bounds so far. In some cases we can show that
the upper bound for $\cA_q(n,k,t;\lambda)$, e.g.\ obtained by linear programming methods, see Section~\ref{sec:IntProg}, or some other method, is
also valid for $\cA^r_q(n,k,t;\lambda)$ by some extra consideration. An example is given by $\cA_2(6,4,2;2)=21$ (see Proposition~\ref{prop:propInApp}). If a block $C$ occurs twice in a
$2-(6,4,2)_2$ subspace packing $\cC$, then each $2$-subspace of $C$ is already covered twice. Each further block has to intersect $C$ dimension at
least $2$, so that we have $\#\cC=2$. Since $\cA_2(6,4,2;2)$ is clearly at least $2$, we have $\cA^r_2(6,4,2;2)=\cA_2(6,4,2;2)$. The combination of
$\cA^r_2(6,4,2;2)\le 21$ with Proposition~\ref{prop_intersection_recursion} gives $\cA_2(8,6,2;3)\le 22$. While we can show $\cA_2(5,3,2;2)=32$
using integer linear programming methods, the subsequent
Proposition~\ref{prop_quadratic_bound_1} gives $\cA_2(5,3,2;2)\le \cA^r_2(5,3,2;2)\le 33$, which
then implies $\cA_2(7,5,2;3)\le \cA^r_2(7,5,2;3)\le 34$.

\medskip

For our next upper bound the underlying approach is based on the second-order Bonferroni Inequality, see e.g.\
\cite{MR3543542} for an application on mixed-dimension subspace codes. It was also used in the derivation of the Drake-Freeman bound
for partial spreads \cite{nets_and_spreads}, cf.\ \cite[Theorem 2.10]{kurz2017packing}.  We first give a technical auxiliary result.

\begin{lemma}
	\label{lemma_standard_equations_special1}
	Let $a_i$ be a non-negative number for each integer $i \geq 0$. If there exist numbers $\mu_0,\mu_1,\mu_2$ and a
	positive integer $m$ such that $\sum_{i\ge 0} a_i=\mu_0$,
	$\sum_{i\ge 0} ia_i=\mu_1 c$, $\sum_{i\ge 0} i(i-1)a_i \leq \mu_2c$, and $2m\mu_1>\mu_2$ then $c\le \frac{m(m+1)\mu_0}{2m\mu_1-\mu_2}$.
\end{lemma}
\begin{proof}
	Let $m$ be an arbitrary integer, then
	$$
	m(m+1)\sum_{i\ge 0} a_i -2m \sum_{i\ge 0} ia_i + \sum_{i\ge 0} i(i-1)a_i  \leq m(m+1)\mu_0 -2m \mu_1 c + \mu_2c
	$$
	which implies that
	$$
	\sum_{i\ge 0} (i-m)(i-m-1)a_i\le m(m+1)\mu_0-2m\mu_1c+\mu_2c.
	$$
	Since $\sum_{i\ge 0} (i-m)(i-m-1)a_i \geq 0$, the last inequality is reduced to
	$$
	0 \leq m(m+1)\mu_0-2m\mu_1c+\mu_2c,
	$$
	which implies that
	$$
	c \leq \frac{m(m+1)\mu_0}{2m\mu_1-\mu_2}~.
	$$
\end{proof}
Minimizing the upper bound for $c$ in Lemma~\ref{lemma_standard_equations_special1}
as a function of $m$ induces $m=\frac{\mu_2\pm\sqrt{\mu_2^2+\mu_2}}{2\mu_1}$. Assuming $\mu_1>0$, $\mu_2\ge 0$,
the optimal choice would be $m=\frac{\mu_2+\sqrt{\mu_2^2+\mu_2}}{2\mu_1}$ since we have to satisfy $2m\mu_1>\mu_2$. Moreover, $m$ has
to be an integer, so that $m=\left\lceil \frac{\mu_2+\sqrt{\mu_2^2+\mu_2}}{2\mu_1} \right\rceil$ is a good choice.
One may also try rounding down.

\begin{proposition}
	\label{prop_quadratic_bound_1}
	If $2(q+1)m > \sbinomq{n-2}{1}$ for a positive integer $m$ and $n \geq 3$, then
	$$\cA^r_q(n,n-2,n-3;2)\le \left\lfloor \sbinomq{n}{1}\cdot\frac{m(m+1)}{2(q+1)m-\sbinomq{n-2}{1}}\right\rfloor.$$
\end{proposition}
\begin{proof}
	Let $\cC$ be a subspace packing with $\cA^r_q(n,n-2,n-3;2)$ blocks and for each $i \geq 1$ let $a_i$ denote
	the number of $(n-1)$-subspaces (hyperplanes) of $\F_q^n$
	containing exactly $i$ blocks of $\cC$. Since there are $\sbinomq{n}{1}$ distinct
	$(n-1)$-subspaces we clearly have
	$$
	\sum_{i\ge 0} a_i = \sbinomq{n}{1}~.
	$$
	Each block $X$ is an $(n-2)$-subspace and hence it is contained in $\sbinomq{2}{1}$ hyperplanes.
	On the other hand summing the number of blocks in all the $(n-1)$-subspaces with repetitions
	is $\sum_{i\ge 1} ia_i$ and hence we have
	$$
	\sum_{i\ge 0} ia_i = \sbinomq{2}{1} \cA^r_q(n,n-2,n-3;2)~.
	$$
	The number of ordered pairs of blocks from $\cC$ which are contained in a given hyperplane~$H$
	which contains exactly $i$ codewords is $i(i-1)$. Hence, the number of ordered pairs of blocks
	which are contained in the same hyperplane with $i$ blocks is $i(i-1) a_i$. Therefore, the number of
	such ordered pairs in all $(n-1)$-subspaces of $\F_q^n$ is $\sum_{i\ge 0} i(i-1)a_i$.
	For a given block $X$ of dimension $n-2$, the number of other blocks which intersect $X$
	in an $(n-3)$-subspace is at most $\sbinomq{n-2}{n-3}=\sbinomq{n-2}{1}$ since any $(n-3)$-subspace can be contained in at
	most $\lambda =2$ blocks. Each two blocks which are contained in the same $(n-1)$-subspace intersect
	in exactly an $(n-3)$-subspace. Hence, the number of ordered pair in all the hyperplanes is at most
	$\sbinomq{n-2}{1}\cA^r_q(n,n-2,n-3;2)$. Therefore, we have
	$$
	\sum_{i\ge 0} i(i-1)a_i  \leq \sbinomq{n-2}{1} A^r_q(n,n-2,n-3;2).
	$$
	Thus, we can apply Lemma~\ref{lemma_standard_equations_special1} with $\mu_0=\sbinomq{n}{1}$,
	$\mu_1=\sbinomq{2}{1}=q+1$, and $\mu_2=\sbinomq{n-2}{1}$; and obtain the claim of the proposition. (Note that $2m\mu_1>\mu_2$.)
\end{proof}

We can apply Proposition~\ref{prop_quadratic_bound_1} in many cases.
For example, by choosing $m=3$ we obtain $A^r_2(5,3,2;2)\leq 33$ and by choosing $m=6$ we obtain $A^r_2(6,4,3;2)\le 126$.
For $m=11$ we obtain $A^r_2(7,5,4;2)\leq 478$ and for $m=21$ or $m=22$ we obtain $A^r_2(8,6,5;2)\leq 1870$.
This method can be extended for other values of $\lambda$ greater than 2.
For $\lambda =2$, the essential step is the determination of a suitable upper bound on $\mu_2$,
as $2m\mu_1>\mu_2$.

Of course we can also apply integer linear programming techniques in order to obtain upper bounds for
$\cA_q(n,k,t;\lambda)$ (or $\cA^r_q(n,k,t;\lambda)$), see Section~\ref{sec:IntProg}.

Another special case occurs if the dimension $k$ of the blocks is almost as large as the dimension $n$ of the ambient space,
i.e., $k=n-1$. The first non-trivial parameters are $\cA_q(3,2,1;\lambda)$ (for $\lambda>1$). In geometrical terms we ask for the maximum number
of lines in $\mathbb{F}_q^3$ such that every point is covered at most $\lambda$ times. Via dualizing, this is equivalent to the maximum number of points in $\mathbb{F}_q^3$ such that every line contains at most $\lambda$ points. The extremal configurations
are also called $(c,\lambda)$-arcs in $\operatorname{PG}(2,q)$, where $c=\cA_q(3,2,1;\lambda)$. More generally, an $(c,\lambda)$-arc
in $\operatorname{PG}(n-1,q)\simeq\mathbb{F}_q^n$ is a set of $c$ points (of $\mathbb{F}_q^n$) such that every hyperplane contains
at most $\lambda$ points (and there is one hyperplane containing exactly $\lambda$ points). Dualized again, the maximum possible
value for $c$ coincides with $\cA_q(n,n-1,1;\lambda)$. Taking the points of an arc as columns of a generator matrix of a linear code
we see, that an $(c,c-d)$-arc in $\mathbb{F}_q^n$ is equivalent to a projective, i.e., any two columns of the generator matrix are
linearly independent, linear $[c,n,d]$-code. Naturally, a lot of knowledge on the maximum size of arcs can be found in the literature.
Several values are known exactly, while only lower and upper bounds are known if the field size $q$ or $\lambda$ increases, see e.g.\
\cite{ball2005bounds}. As a well-known result we remark $\cA_q(3,2,1;\lambda)=q+2$ for even $q$ and $\cA_q(3,2,1;\lambda)=q+1$ otherwise.

\vspace{2ex}
\section{Constructions for Subspace Packings}
\vspace{-.25ex}
\label{sec:constructions}

Here we will study more sophisticated construction methods for subspace packings. In \cite{heinlein2016tables} the
authors also study which of the known constructions for constant-dimension codes yield the currently best known lower
bounds for $\cA_q(n,k,t;1)$ in the most number of cases. The two most
successful approaches are the echelon-Ferrers Construction (including their different variants) and
the so-called \emph{linkage construction} \cite{MR3543532}. We remark that improvements of the original
linkage construction were obtained in \cite{heinlein2017asymptotic,kurz2019linkage}. In Subsection~\ref{subsec_linkage}
a generalization of the linkage construction for $\lambda>1$ will be presented. For small parameters larger constant-dimension
codes were also constructed using an integer linear programming formulation and the prescription of automorphisms, see
e.g.\ \cite{kohnert2008construction}. We will adjust this method in Subsection~\ref{sec:IntProg}. Some tailored constructions that
indeed meet the known upper bounds are stated in Subsection~\ref{subsec:exact}. $q$-analogs of group divisible designs also give
some good constructions for a few parameters, see~\cite{BKKNW}. Of course a packing design is the best that can be achieved,
so that we also refer to the corresponding literature, see e.g.\ \cite{braun2018q}.

\subsection{A variant of the linkage construction}
\label{subsec_linkage}

An $\alpha-(n,k,\delta)_q^c$ covering Grassmanian
code $\cC$ consists of a set of $k$-subspaces of $\F_q^n$ such that every set of $\alpha$ codewords span
a subspace of dimension at least $\delta+k$. The maximum size of a related code is denoted by $\cB_q(n,k,\delta;\alpha)$.
It was proved in~\cite{EtZh18} that
\[
\cA_q(n,k,t;\lambda)=\cB_q(n,n-k,k-t+1;\lambda+1)~,
\]
and
$$
\cB_q(n,k,\delta;\alpha)=\cA_q(n,n-k,n-k-\delta+1;\alpha-1).
$$

Finally, we will use a simple connection between the subspace distance of two $k$-subspaces $U$ and $V$ of $\F_q^n$,
and a related rank for the row space of these two subspaces
$$
d_S (U,V)=2\dim(U+W)-\dim(U)-\dim(V)=2\left(\operatorname{rk}\begin{pmatrix}\tau(U)\\\tau(V)\end{pmatrix}-k\right).
$$
Here $\tau(U)$ and $\tau(V)$ are $k \times n$ matrices over $\F_q$ whose row spaces are $U$ and $V$.
Similarly, if $U$ and $V$ arise from lifting two matrices $M_1$ and $M_2$, i.e. they are of the form $U=\mathrm{rowspace}(I_k|M_1)$ and $V=\mathrm{rowspace}(I_k|M_2)$, then
$$d_S(U,V)\ge 2\operatorname{rk}(M_1-M_2)=2 d_R (M_1,M_2).$$

\begin{theorem}
	\label{thm:linkage}
	Let $1 \leq \delta \leq k$, $k+\delta\leq n$ and $2\leq \alpha \leq q^k+1$ be integers.
	\begin{enumerate}
		\item\label{1} If $n<k+2\delta,$ then
		$$\cB_q(n,k,\delta ;\alpha)\geq (\alpha -1)q^{\max\{k,n-k\}(\min\{k,n-k\}-\delta+1)}.$$
		
		\item\label{2} If $n\geq k+2\delta,$ then for each $t$ such that $\delta \leq t \leq n-k-\delta$, we have
		\begin{enumerate}
			\item\label{2a} If $t<k$, then			
			$$\cB_q(n,k,\delta ;\alpha)\geq (\alpha -1)q^{k(t-\delta +1)} \cB_q(n-t,k,\delta ;\alpha).$$
			
			\item\label{2b} If $t\geq k$, then 			
			$$\cB_q(n,k,\delta ;\alpha)\geq (\alpha -1)q^{t(k-\delta +1)} \cB_q(n-t,k,\delta ;\alpha)+\cB_q(t+k-\delta,k,\delta;\alpha).$$
		\end{enumerate}
	\end{enumerate}
\end{theorem}

\begin{remark}
	Note that the length of vectors is expected to be greater than or equal to $ k+\delta$. However, in Case \ref{2b} of Theorem \ref{thm:linkage}, there is a possibility that $t+k-\delta<k+\delta$ for $\cB_q(t+k-\delta,k,\delta;\alpha)$. In such situations, we consider the following convention:
	$$
	\cB_q(t+k-\delta,k,\delta;\alpha)=\min\left\{\alpha-1, {t+k-\delta \brack k}_q  \right\}.
	$$
\end{remark}

\subsubsection*{Proof of Theorem \ref{thm:linkage}}

The proof of Theorem~\ref{thm:linkage} will be in a few steps. \bigbreak
\noindent
{\bf Case \ref{1}: $k+\delta \leq n<k+2\delta$}
\begin{construction}
	\label{const1}
	Let $I_k$ denote the $k\times k$ identity matrix over $\mathbb{F}_q$ and let $C_1\subseteq \mathbb{F}_q^{k\times (n-k)}$
	be a linear MRD code with minimum rank distance $\delta$. Let $C_1, C_2,\dots ,C_{\alpha -1}$
	be $\alpha -1$ pairwise disjoint MRD codes with minimum rank
	distance $\delta$ obtained by translating $C_1$ in a way that (see~\cite{EtSi13})
	$
	d_R(C_1\cup\dots\cup C_{\alpha -1})=\delta -1
	$.
	Let $C \triangleq C_1\cup\dots\cup C_{\alpha -1}$. Lifting the matrices in $C$, i.e. concatenating each matrix to the $k\times k$ identity matrix $I_k$,
	$$(\alpha -1) q^{\max\{k,n-k\}(\min\{k,n-k\}-\delta +1)}$$ different matrices
	of size $k\times n$, in reduced row echelon form (RREF in short), are constructed.
	Let $\mathrm{RREF}(\C)$ denote the set of these matrices, and let $\C$
	be the set of rowspaces of matrices in $\mathrm{RREF}(\C)$.
\end{construction}

\begin{claim}
	Let $\C$ be the set of $k$-subspaces obtained in Construction~\ref{const1}. Then we have
	$$
	\dim (U_1+\dots +U_{\alpha})\geq k+\delta,
	$$
	for each $\alpha$ distinct codewords $U_1,\dots ,U_\alpha \in \C$.
\end{claim}

\begin{proof}
	Given $\alpha$ distinct codewords $U_1,\dots,U_{\alpha}\in \C$,
	let $u_1,\dots,u_{\alpha}\in \mathrm{RREF}(\C)$ be the corresponding $k\times n$ matrices
	in RREF. Let $A_1,\dots ,A_\alpha$ be the $\alpha$ distinct codewords of $C$ satisfying
	$
	U_i
	=\mathrm{rowspace}(I_k|A_i)
	$
	for each $1\leq i\leq \alpha$. For these $\alpha$ codewords of $\C$ we have that $\dim (U_1+\dots +U_{\alpha})$ is equal to
	the rank of the $(\alpha k) \times n$ related matrix, i.e.
	\begin{equation}\label{eq:rank-1}
	\mathrm{rank}
	\begin{tabular}{l}
	\begin{tabular}{|l|l|}	
	\hline $I_k$ & $A_1$\\
	\hline $I_k$ & $A_2$\\
	\hline $\vdots$ & \vdots \\
	\hline $I_k$ & $A_\alpha$  \\ \hline
	\end{tabular}
	\end{tabular}. 
	\end{equation}
	Note that $A_1,\dots ,A_\alpha\in C=C_1\cup\dots\cup C_{\alpha-1}$, i.e. at least two of $A_i$'s
	must be from the same rank-metric code $C_j$ for some $1\leq j\leq \alpha-1$. W.l.o.g., assume $A_1$ and $A_2$ are
	from the same code $C_j$ for some $1\leq j\leq \alpha-1$. Clearly (\ref{eq:rank-1}) is equal to
	\[
	\mathrm{rank}
	\begin{tabular}{l}
	\begin{tabular}{|l|l|}	
	\hline $I_k$ & $A_1$\\
	\hline $0$ & $A_2-A_1$\\
	\hline $\vdots$ & \vdots \\
	\hline $0$ & $A_\alpha-A_1$  \\ \hline
	\end{tabular}
	\end{tabular}\geq \mathrm{rank}
	\begin{tabular}{l}
	\begin{tabular}{|l|l|}	
	\hline $I_k$ & $A_1$\\
	\hline $0$ & $A_2-A_1$\\ \hline
	\end{tabular}
	\end{tabular}\geq k+\delta.
	\]
\end{proof}
\bigbreak
\noindent
{\bf Case \ref{2a}: $k+2\delta \leq n$, $t\leq n-k-\delta$, and $\delta \leq t < k$}
\bigbreak
\begin{construction}
	\label{const2}
	Let $\C_{n-t}$ be a set of $k$-subspaces of $\F_q^{n-t}$ such that any
	$\alpha$ distinct $k$-subspaces $V_1,\dots ,V_{\alpha}\in\C_{n-t}$ satisfy $\dim (V_1+\dots +V_{\alpha})\geq k+\delta$,
	and $|\C_{n-t}|=B_q(n-t,k,\delta ;\alpha)$ (note that $n-t\geq k+\delta$).
	\begin{enumerate}
		\item For each $V \in \C_{n-t}$, let $v \in \F_q^{k\times (n-t)}$ be the unique matrix in RREF such
		that $V$ is the rowspace of $v$. The set $\mathrm{RREF}(\C_{n-t})$ contains all the subspaces of $\C_{n-t}$ in this form.
		
		\item Let $C_1\subseteq \F_q^{k\times t}$ be a linear MRD code with minimum rank distance $\delta$.
		Let $C_1,C_2,\dots ,C_{\alpha -1}$ be $\alpha -1$ pairwise disjoint MRD codes with minimum rank distance $\delta$
		obtained by translating $C_1$ in a way that (see~\cite{EtSi13})
		$$
		d_R(C_1\cup\dots\cup C_{\alpha -1})=\delta -1.
		$$
		Let $C \triangleq C_1\cup\dots\cup C_{\alpha -1}$. By concatenating each matrix in $C$ to the end of each
		$u\in\mathrm{RREF}(\C_{n-t})$, $(\alpha -1) q^{k(t-\delta+1)}|\C_{n-t}|$
		different matrices, of size $k\times n$, in RREF are constructed. Let $\mathrm{RREF}(\C)$ denote the set of these matrices,
		whose rowspaces form the code $\C$.
	\end{enumerate}
\end{construction}

\begin{claim}
	\label{claim2}
	If $\C$ is the set of $k$-subspaces in Construction \ref{const2}, then
	$$
	\dim (U_1+\dots +U_{\alpha})\geq k+\delta,
	$$
	for each $\alpha$ distinct codewords $U_1,\dots ,U_{\alpha}$ of $\C$.
\end{claim}
\begin{proof}
	Given $\alpha$ distinct codewords $U_1,\dots,U_{\alpha}$ of $\C$, let $u_1,\dots,u_{\alpha}\in \mathrm{RREF}(\C)$
	be the corresponding $k\times n$ matrices in RREF. Let
	$v_1,\dots ,v_{\alpha}\in \mathrm{RREF}(\C_{n-t})$ and $A_1,\dots ,A_\alpha$ be $\alpha$ codewords of $C$ satisfying
	$$
	U_i=\mathrm{rowspace}(u_i)=\mathrm{rowspace}([v_i|A_i])
	$$
	for each $1\leq i\leq \alpha$. Clearly, $\dim (U_1+\dots +U_{\alpha})$ is equal to
	\begin{equation}
	\label{eq:rank-2}
	\mathrm{rank}
	\begin{tabular}{l}
	\begin{tabular}{|l|l|}	
	\hline $v_1$ & $A_1$\\
	\hline $v_2$ & $A_2$\\
	\hline $\vdots$ & \vdots \\
	\hline $v_\alpha$ & $A_\alpha$  \\ \hline
	\end{tabular}
	\end{tabular}~. 
	\end{equation}
	We distinguish between three cases.
	\begin{itemize}
		\item \textbf{Case A.} If $v_1=v_2=\dots =v_\alpha$, then $A_1,\dots ,A_\alpha$ are different matrices. Note that
		$A_1,\dots ,A_\alpha\in C=C_1\cup\dots\cup C_{\alpha-1}$, which implies that at least two of the $A_i$'s must be from the
		same rank-metric code $C_j$ for some $1\leq j\leq \alpha-1$. W.l.o.g., assume $A_1$ and $A_2$ are from
		the code $C_j$ for some $1\leq j\leq \alpha-1$. Then clearly (\ref{eq:rank-2}) is equal to
		$$
		\mathrm{rank}
		\begin{tabular}{l}
		\begin{tabular}{|l|l|}	
		\hline $v_1$ & $A_1$\\
		\hline $0$ & $A_2-A_1$\\
		\hline $\vdots$ & \vdots \\
		\hline $0$ & $A_\alpha-A_1$  \\ \hline
		\end{tabular}
		\end{tabular}\geq \mathrm{rank}
		\begin{tabular}{l}
		\begin{tabular}{|l|l|}	
		\hline $v_1$ & $A_1$\\
		\hline $0$ & $A_2-A_1$\\ \hline
		\end{tabular}
		\end{tabular}\geq k+\delta.
		$$
		
		\item \textbf{Case B.} Assume $v_i\neq v_j$ for all $1\leq i<j\leq \alpha$. In this case,
		\[
		\begin{array}{lll}
		\mathrm{rank}
		\begin{tabular}{l}
		\begin{tabular}{|l|l|}	
		\hline $v_1$ & $A_1$\\
		\hline $v_2$ & $A_2$\\
		\hline $\vdots$ & \vdots \\
		\hline $v_\alpha$ & $A_\alpha$  \\ \hline
		\end{tabular}
		\end{tabular} 
		& \geq &
		\mathrm{rank}
		\begin{tabular}{l}
		\begin{tabular}{|l|}	
		\hline $v_1$ \\
		\hline $v_2$ \\
		\hline $\vdots$  \\
		\hline $v_\alpha$ \\ \hline
		\end{tabular}
		\end{tabular}\\ 
		\ & \ & \ \\
		\ & = & \dim (\mathrm{rowspace}(v_1)+\dots +\mathrm{rowspace}(v_\alpha)) \\
		\ & \ & \ \\
		\ & \geq & k+\delta
		\end{array}		
		\]
		by the definition of $\C_{n-t}$.
		
		\item \textbf{Case C.} The only remaining case is that some of the $v_i$'s are different and some are equal.
		W.l.o.g. assume that $v_1\neq v_2 =v_3$ which implies $A_2\neq A_3$. Hence, (\ref{eq:rank-2}) equals to		
		$$
		\begin{array}{lll}
		\mathrm{rank}
		\begin{tabular}{l}
		\begin{tabular}{|l|l|}	
		\hline $v_1$ & $A_1$\\
		\hline $v_2$ & $A_2$\\
		\hline $0$ & $A_3-A_2$\\
		\hline $\vdots$ & \vdots \\
		\hline $v_\alpha$ & $A_\alpha$  \\ \hline
		\end{tabular}
		\end{tabular}
		& \geq &  \mathrm{rank}
		\begin{tabular}{l}
		\begin{tabular}{|l|l|}	
		\hline $v_1$ & $A_1$\\
		\hline $v_2$ & $A_2$\\
		\hline $0$ & $A_3-A_2$\\
		\hline
		\end{tabular}
		\end{tabular}
		\\ \ & \geq & \mathrm{rank}
		\begin{tabular}{l}
		\begin{tabular}{|l|}	
		\hline $v_1$ \\
		\hline $v_2$ \\
		\hline
		\end{tabular}
		\end{tabular} +\mathrm{rank}(A_3-A_2)
		\\ \ & \ & \\ \ & \geq & (k+1)+(\delta -1) 
		\\  \ & \ & \\ \ & = & k+\delta.
		\end{array}
		$$
	\end{itemize}
\end{proof}
\bigbreak
\noindent
{\bf Case \ref{2b}: $k+2\delta \leq n$ and $k \leq t\leq n-k-\delta$}
\bigbreak
\begin{construction}
	\label{const3}
	Let $\C_{n-t}$ be a set of $k$-subspaces of $\F_q^{n-t}$ such that any
	$\alpha$ distinct $k$-subspaces $U_1,\dots ,U_{\alpha}\in\C_{n-t}$ satisfy $\dim (U_1+\dots +U_{\alpha})\geq k+\delta$, and $|\C_{n-t}|=B_q(n-t,k,\delta ;\alpha)$ (note that $n-t\geq k+\delta$).
	\begin{enumerate}
		\item For each $U \in \C_{n-t}$, let $u \in \F_q^{k\times (n-t)}$ be the unique matrix in RREF such
		that $U$ is the rowspace of $u$. The set $\mathrm{RREF}(\C_{n-t})$ contains all the subspaces of $\C_{n-t}$ in this form.
		
		\item Let $C_1\subseteq \F_q^{k\times t}$ be a linear MRD code with minimum rank distance $\delta$.
		Let $C_1,C_2,\dots ,C_{\alpha -1}$ be the $\alpha -1$ pairwise disjoint MRD codes of minimum rank distance $\delta$
		obtained by translating $C_1$ in a way that (see~\cite{EtSi13})
		$$
		d_R(C_1\cup\dots\cup C_{\alpha -1})=\delta -1.
		$$
		Let $C \triangleq C_1\cup\dots\cup C_{\alpha -1}$. By concatenating each matrix in $C$ to the end of each matrix
		$u\in\mathrm{RREF}(\C_{n-t})$, $(\alpha -1) q^{t(k-\delta+1)}|\cC_{n-t}|$
		different matrices, of size $k\times n$, in RREF are constructed. Let $\mathrm{RREF}(\C)$ denote the set of these matrices,
		whose rowspaces form the code $\C$.
		
		\item Consider a code $\C_{\mathrm{app}}\subseteq \cG_q(n,k)$ such that
		\begin{itemize}
			\item the first $n-(t+k-\delta)$ entries of each codeword in $\C_{\mathrm{app}}$ are \emph{zeroes},
			\item Each $\alpha$ distinct codewords $U_1,\dots,U_\alpha$ of $\C_{\mathrm{app}}$, satisfy $\dim (U_1+\dots+U_\alpha)\geq k+\delta$.
			\item $\C_{\mathrm{app}}$ is of maximum size, i.e. $|\C_{\mathrm{app}}|=\cB_q(t+k-\delta,k,\delta;\alpha)$.
		\end{itemize}		
	\end{enumerate}
	
	Form a new code $\C'$ as the union of $\C$ in Step 2 and $\C_{\mathrm{app}}$ in Step 3.
\end{construction}

\begin{claim}
	If $\C'$ is the set of $k$-subspaces in Construction~\ref{const3} and $U_1,\ldots,U_\alpha$ are $\alpha$
	distinct codewords of $\C'$, then
	$$
	\dim (U_1+\dots +U_{\alpha})\geq k+\delta.
	$$
\end{claim}
\vspace{-0.3cm}
\begin{proof}
	The first two steps of Construction~\ref{const3} are the same as the ones in Construction~\ref{const2}. Therefore, the Claim
	follows from the proof of the claim after Construction~\ref{const2} and the definition of $\C_{\mathrm{app}}$ in Construction~\ref{const3}.
\end{proof}
\vspace{-0.3cm}
\begin{corollary}
	\label{cor:linkage}
	Let $1\leq s \leq k\leq n$ and $1\leq \lambda \leq q^k$ be integers.
	\begin{enumerate}
		\item\label{1} If $k>2t-2,$ then
		\[\cA_q(n,k,t ;\lambda)\geq \lambda q^{\max\{k,n-k\}(\min\{k,n-k\}-k+t)}. \]
		
		\item\label{2} If $k\leq 2t-2,$ then choosing an arbitrary $s$ satisfying $k-t+1 \leq s \leq t-1,$ we have that	
		\begin{enumerate}
			\item\label{2a} If $s<n-k$, then			
			\[\cA_q(n,k,t ;\lambda)\geq \lambda q^{(n-k)(s-k +t)} \cA_q(n-s,k-s,t-s ;\lambda).\]
			
			\item\label{2b} If $s\geq n-k$, then 			
			$$\cA_q(n,k,t ;\lambda)\geq \lambda q^{t(n-2k+t)} \cA_q(n-s,k-s,t-s;\lambda)$$
			$$+\cA_q(s+n-2k+t-1,s-k+t-1,s-2k-2t-1;\lambda).$$
		\end{enumerate}
	\end{enumerate}
	\bigbreak
\end{corollary}

\subsection{Integer Linear Programming lower bounds}
\vspace{-.25ex}
\label{sec:IntProg}

The problem of the determination of $\cA_q(n,k,t;\lambda)$ can be formulated as an integer linear programming problem. For $\lambda=1$
the reader is referred to~\cite{kohnert2008construction}. For each $k$-subspace $U$ of
$\F_q^n$ a binary variable $x_U$ is defined. (For $\cA^r_q(n,k,t;\lambda)$ we use $x_U\in\mathbb{N}$.) The value of this variables is \emph{one} if $U$ is contained
in the subspace packing and \emph{zero} if $U$
is not contained in the subspace packing. (In general, $x_U$ is the number of times the subspace $U$ is contained as a block in the corresponding
subspace packing.) The set of inequalities will be called \emph{extensive formulation}
since it contains a huge number of variables and constraints:
\begin{eqnarray}
\max \sum_{U \in \cG_q(n,k)} x_U\label{ILP_formulation}\\
\text{subject~to}~~~~~~\nonumber\\
\forall V \in \cG_q(n,t) ~ \sum_{V \subset U \in \cG_q(n,k)} x_U &\le& \lambda\nonumber\\
\forall 1\le i<t, W \in \cG_q(n,i) ~ \sum_{W\le U\le \F_q^n\,:\, \dim(U)=k} x_U &\le& \cA_q(n-i,k-i,t-i;\lambda), \nonumber\\
\text{where} ~ x_U &\in& \{0,1\},~\text{for~each} ~ U \in \cG_q(n,k)\nonumber
\end{eqnarray}

The second set of constraints, i.e., those for $1\le i\le t-1$, are not necessary to guarantee that the maximum target value
equals $\cA_q(n,k,t;\lambda)$, but they may significantly speed up the computation.
However, this integer linear programming formulation can be solved exactly just
for rather small parameters due to the exponential number of variables and constraints.

As for the case of constant-dimension codes, i.e., $\lambda=1$ with $\cA_2(6,3,2;1)=77$ \cite{hkk77} and
$\cA_2(8,4,2;1)=257$ \cite{heinlein2017classifying}, some of the best known upper bounds are so far only
obtained via integer linear programming, see 
the appendix. An example
is $\cA_2(5,3,2;2)=32$, where Proposition~\ref{prop_quadratic_bound_1} (with $q=2$, $n=5$,
and $m=3$) gives $\cA_2(5,3,2;2)\le 33$. We remark that the LP relaxation, i.e., if we replace $x_U\in\{0,1\}$ by
$0\le x_U\le 1$, of the above ILP is not very good. More precisely, if we do not use the second set of constraints, then
we end up with the packing bound of Proposition~\ref{prop_packing}.
\begin{proposition}\label{prop:propInApp}
		\label{proposition_a_2_6_4_2_2}
		$\cA_2(6,4,2;2)=21$
	\end{proposition}
	\begin{proof}
		Let $\cC$ be a $2-(6,4,2)_2$ subspace packing. 
		Any two solids in $\cC$ intersect either in dimension $2$ or dimension $3$. If any pair of solids intersects in dimension $3$, then
		$\#\mathcal{C}\le 2$ since two planes contained in a solid intersect in a line. 
		
		So, let $U_1$ and $U_2$ be two arbitrary solids intersecting 
		in a line. Up to symmetry there is only one choice. Now let $U_3$ be another solid intersecting $U_1$ and $U_2$ in a line such that 
		$\dim(U_1\cap U_2\cap U_3)=0$. 
		Again there is a unique choice up to isomorphism. (This fact may be checked directly since the parameters
		are quite small. Alternatively one can characterize triples of subspaces uniquely by the numbers of the dimensions of all possible
		unions and intersections.) 
		
      We extend the integer linear programming formulation from (\ref{ILP_formulation}) and prescribe $U_1$, $U_2$, and $U_3$, i.e., 
      we additionally set $x_{U_1}=1$, $x_{U_2}=1$, and $x_{U_3}=1$. This ILP model was solved
		after a week of computation time 
		with optimal target value $21$. 
		
		The action of the stabilizer of $\{U_1,U_2\}$ on the set of solids with the intersections described
		above gives an orbit $\left\{U_3^1,\dots,U_3^{256}\right\}$ of length $256$. Prescribing $U_1$, $U_2$ and excluding the corresponding $256$ choices, i.e., 
      starting from (\ref{ILP_formulation}) and additionally setting $x_{U_1}=1$, $x_{U_2}=1$, and $x_{U_3^i}=0$ for all $1\le i\le 256$ gives an ILP formulation  
		whose LP relaxation was solved in less than a second with target value $20$. Thus, $\cA_2(6,4,2;2)\le 21$.
		
		For the lower bound we consider a line spread $\mathcal{L}$ of $\mathbb{F}_2^6$ such that any three lines generate a subspace of dimension at
		least $5$. The dual of $\mathcal{L}$ is a set of $21$ solid such that no three solids intersect in a line. It can be easily checked
		that those special line spreads exist.
	\end{proof}
We remark that all line spreads in $\mathbb{F}_2^6$ have been classified in \cite{mateva2009line}. The line spreads used in the construction of
Proposition~\ref{proposition_a_2_6_4_2_2} are kind of the opposite of geometric line spreads, where any three lines either generate a solid or the
full ambient space.

If we are not interested in the exact value of $\cA_q(n,k,t;\lambda)$ but good lower bounds, then prescribing some automorphisms for
subspace packings can reduce the number of variables and constraints to a manageable size also for larger parameters, see e.g.\
\cite{kohnert2008construction} for the application of this technique to constant-dimension codes. An example verifying
$\cA_2(7,3,2;2)\ge 741$ was found prescribing a Heisenberg group of order $27$. 
Going over to a subgroup of order nine gives $\cA_2(7,4,2;2)\ge 96$. 
Again the Heisenberg group of order $27$ gives $\cA_2(7,4,3;2)\ge 906$ and $\cA_2(7,5,4;2)\ge 360$. 

\subsection{Exact sizes of packings}
\vspace{-.25ex}
\label{subsec:exact}

For $\lambda =1$ we have already mentioned that the exact value
of $\cA_q(n,k,t;1)$ can be derived if we know the size of the largest $(n,2(k-t+1),k)_q$ code.
Unfortunately, this is known in a small number of cases. For larger $\lambda$ this is
fortunately better. When a $t-(n,k,\lambda)_q$ design exists, the number of blocks in the design
is exactly the value of $\cA_q(n,k,t;\lambda)$. Many such designs are known and their
parameters are summarized in~\cite{braun2018q}. We have also the following result. 
\begin{theorem}
	\label{thm:disj_designs}
	If there exists a set of $s$ pairwise disjoint $t-(n,k,\lambda)_q$ designs then
	we have $\cA_q(n,k,t;\lambda j) =\lambda j\cdot\sbinomq{n}{t}/\sbinomq{k}{t}$. for each $1 \leq j \leq s$.
\end{theorem}
Theorem~\ref{thm:disj_designs} can be applied for a limited number of parameters.
The best are based on partitioning of all $k$-subspaces into such designs as
discussed in~\cite{BKKL,BKOW,HKM16,KLW18}. There are other with smaller $t$, especially when $t=1$.
In this special case we consider a \emph{$(k-1)$-parallelism} in
$\mathbb{F}_q^n$, which is a partition of the set of $k$-subspaces into ($\gaussm{n}{k}{q}\cdot \gaussm{k}{1}{q}/\gaussm{n}{1}{q}$)
$k$-spreads (Recall that this is in the language of vector spaces).
In general, parallelisms are a well known concept for combinatorial designs.
In the $q$-analog case not so many
examples are known. $2$-parallelism exist e.g.\ for $q=2$ and all even $n$ \cite{Bak76,wettl1991parallelisms} or
for any prime power $q$ if $n=2^i$ for $i\ge 2$ \cite{Beu90}, see also \cite{denniston1972some} for the case $i=2$.
Another example for $\mathbb{F}_3^6$ was found in \cite{EtVa12}. A $3$-parallelism in $\mathbb{F}_2^6$ was found
in \cite{hishida2000cyclic,Sar02}. All such examples with an automorphism
group of order $31$ are classified in~\cite{ToZh10}. Similar results can be obtained by using disjoint subspace packings.

\begin{proposition}
	\label{prop_parallel_packing}
	If there exists a set of $s$ pairwise disjoint $\spparam$ subspace packings of cardinality $\cA_q(n,k,t;\lambda)$ then
	${\cA_q(n,k,t;s\cdot\lambda) \ge s\cdot \cA_q(n,k,t;\lambda)}$.
\end{proposition}
Beutelspacher proved in \cite{Beu90} that there exist
$\gaussm{2\left\lfloor \log_2 (n-1)\right\rfloor+1}{1}{q}$ pairwise disjoint $2$-spreads in $\mathbb{F}_q^n$ for even $n$.
For larger $k$ this was generalized for the binary case in \cite{Etz15}: If $k<n$ and $k$ divides $n$, then there exist at
least $2^k-1$ pairwise disjoint $k$-spreads in $\mathbb{F}_2^n$. One also speaks of partial parallelisms.

\medskip

By the combination of Lemma~\ref{lemma_intersection_upper_bound} and Lemma~\ref{lemma_trivial} we conclude:
\begin{proposition}
	\label{prop_intersection}
	Let $\lambda,n,k,t$ be positive integers with $1\le t\le k\le n$, $1\le \lambda \leq \sbinomq{n-t}{k-t}$, and
	$(\lambda+1)k-\lambda n\ge t$, then $\cA_q(n,k,t;\lambda)= \lambda$.
\end{proposition}

One more value of $\cA_q(n,k,t;\lambda)$ can be inferred
from Lemma~\ref{lemma_upper_bound_exclude_point} and Lemma~\ref{lemma_disjoint_from_point}:
\begin{proposition}
	\label{prop_exact_disjoint_from_point}
	For $n \geq 3$ we have $\cA_q(n,n-1,n-2;q)= q^{n-1}$.
\end{proposition}

Note that optimal examples for the packings which attains the value in Proposition~\ref{prop_exact_disjoint_from_point}
are unique up to isomorphism, i.e., they are all given by the construction in Lemma~\ref{lemma_disjoint_from_point}.

\section{Conclusion and Problems for Future Research}
\vspace{-.25ex}
\label{sec:problems}

Motivated by an application in network coding, subspace packings were considered in this paper.
For a given finite field $\F_q$, three positive integers $n$, $k$, and $t$ such that
$1 \leq t < k < n$, and a positive integer $\lambda$, such that
$1 \leq \lambda \leq \sbinomq{n-t}{k-t}$ the packing number $\cA_q(n,k,t;\lambda)$ is
the maximum number of $k$-subspaces in a $\spparam$ subspace packing.
Such a subspace packing $\cC$ contains $k$-subspaces of the Grassmannian $\cG_q(n,k)$
for which each $t$-subspace of the Grassmannian $\cG_q(n,t)$ is contained in at most $\lambda$ subspaces of $\cC$.
We have considered various construction methods and upper bounds, some new and some based
on the foundations of known construction for $\lambda =1$. We end our exposition with what we consider to
be the most important problem in this context.

When $\lambda =1$ the size of the codes obtained via the various constructions
are close to the upper bounds, i.e. the codes are asymptotically optimal. When $\lambda > 1$
and $k \leq n/2$ the same claim still holds. When $k > n/2$ and $\lambda > 1$ the codes obtained
by our constructions fall short of the upper bounds, unless $k$ is close to $n$.
An example for our weak bounds in this case can be demonstrated for $n=3\ell$, $k=2\ell$, $t=\ell+1$,
and $\lambda=2$. The upper bound for $\cA_q(3\ell,2\ell,\ell+1;2)$ by Proposition~\ref{prop_packing}
is $q^{c t^2}$ for some constant $c$. A probabilistic argument~\cite{Puc18,Rot16,Sch18} yields that this bound is
attained for smaller constant. But, there is no construction which is getting close to this value.
Such a construction for these parameters or similar ones is one of the most important open problems.
This value is also important for solutions of the generalized combination network which shows that
vector network coding outperforms scalar linear network coding on multicast networks with three messages.

In general those parametric series where both $n$ and $k$ depend on some parameter $l$ are interesting, since
they are not covered by the asymptotic results mentioned in Section~\ref{sec:upper}. A specific example is
$\cA_q(2l,l,2;1)$. Having proved $\cA_2(8,4,2;1)=257$, the authors of \cite{heinlein2017classifying} have conjectured
that for $l\ge 4$ (and $q=2$) the exact value of $\cA_q(2l,l,2;1)$ is indeed attained by an LMRD plus an additional
codeword. However, this easy construction is far away from the upper bound given by the packing bound. So, can
better constructions be found? What happens for $q>2$ or more generally for $\cA_q(2l,l,2;\lambda)$?

\section*{Acknowledgments}
The third author would like to thank T\"UB\.ITAK B\.IDEB for the financial support under programs 2211 and 2214/A. The third author also 
thanks the Computer Science Department at the Technion for their warm-hearted hospitality and support during his stay at the university between November 2017 and May 2018.

%
%




	\section*{Appendix: Tables}
	\label{sec_tables}
	
	In this section we collect some numerical results on $\cA_q(n,k,t;\lambda)$, i.e., the tightest lower and upper bounds known to us.
	We will mainly focus on the binary case $q=2$ and small values of $\lambda$ and give just a few tables for $q=3$. We only provide
	results for $\lambda>1$ and refer the interested reader to \url{http://subspacecodes.uni-bayreuth.de} \cite{heinlein2016tables} for
	$\lambda=1$. In order to point to the origin of the bound or an exact formula we use the following abbreviations:
	\begin{itemize}
		\item $^a$: Bounds for arcs, see e.g.\ \cite{ball2005bounds} and the end of Subsection~\ref{subsec_more_upper_bounds}.
		\item $^b$: Take all subspaces, see Lemma~\ref{lemma_all_subspaces}.
		\item $^c$: All subspaces not containing a point, see Proposition~\ref{prop_exact_disjoint_from_point}.
		\item $^g$: Constructions for $q-GDDs$, a $q$-analog of group divisible designs, see~\cite{BKKNW}.
		\item $^h$: Restriction to a hyperplane, see Proposition~\ref{prop_combination_packing_johnson_hyperplane}.
		\item $^i$: Intersection arguments, see Lemma~\ref{lemma_intersection_upper_bound}, Proposition~\ref{prop_intersection}, and Proposition~\ref{prop_intersection_recursion}.
		\item $^{j}$: Improved Johnson bound for points, see Proposition~\ref{johnson_bound_point_improved}.
		\item $^{k}$: Known results for packing designs, see e.g.\ \cite{braun2018q}.
		\item $^l$: Integer linear programming formulations. 
		\item $^p$: Existence of parallel packings, see Theorem~\ref{thm:disj_designs} in connection with the literature on large sets, and
		Proposition~\ref{prop_parallel_packing} in connection with the literature of (partial) parallelisms.
		\item $^q$: The \textit{quadratic} upper bound from Proposition~\ref{prop_quadratic_bound_1} based on the second-order
		Bonferroni Inequality.
		\item $^t$: Integer linear programming formulations with prescribed automorphisms. 
		\item $^x$: Generalized linkage construction, see Theorem~\ref{thm:linkage} and Corollary~\ref{cor:linkage}.
	\end{itemize}
	We remark that $\cA_2(6,3,2;4)\ge 360$, which was obtained in the context of $q$-GDDs \cite{BKKNW}, was also obtained in \cite{EtHo18}.
	The upper bound for $\cA_2(6,4,2;2)$, based on integer linear programming, need a more detailed explanation, see Proposition~\ref{proposition_a_2_6_4_2_2}, 
	which is marked by a $\star$ in the corresponding table.
	For upper bounds marked by $^i$ we refer to the discussion directly after Proposition~\ref{prop_intersection_recursion} for the
	details.
		

	\begin{table}[!htp]
		\begin{center}
			\begin{tabular}{|r|rrr|}
				\hline
				k/t & 1 & 2 & 3\\
				\hline
				2 & $4^a$ & $7^b$ &       \\
				3 & $1^b$ & $1^b$ & $1^b$ \\
				\hline
			\end{tabular}
			\caption{Bounds for $\cA_2(3,k,t;2)$}
		\end{center}
	\end{table}
	
	\begin{table}[!htp]
		\begin{center}
			\begin{tabular}{|r|rrrr|}
				\hline
				k/t & 1 & 2 & 3 & 4 \\
				\hline
				2 & $10^p$ & $35^b$ &        &       \\
				3 & $2^i$  &  $8^c$ & $15^b$ &       \\
				4 & $1^b$  &  $1^b$ & $1^b$  & $1^b$ \\
				\hline
			\end{tabular}
			\caption{Bounds for $\cA_2(4,k,t;2)$}
		\end{center}
	\end{table}
	
	\begin{table}[!htp]
		\begin{center}
			\begin{tabular}{|r|rrrrr|}
				\hline
				k/t & 1 & 2 & 3 & 4 & 5\\
				\hline
				2 & $20^{l,j}$ &$155^b$ &         &        &       \\
				3 & $8^{j,l} $  & $32^l$ & $155^b$ &        &       \\
				4 & $2^i$  &  $2^i$ & $16^c$  & $31^b$ &       \\
				5 & $1^b$  &  $1^b$ & $1^b$   & $1^b$  & $1^b$ \\
				\hline
			\end{tabular}
			\caption{Bounds for $\cA_2(5,k,t;2)$}
		\end{center}
	\end{table}
	
	\begin{table}[!htp]
		\begin{center}
			\begin{tabular}{|r|rrrrr|}
				\hline
				k/t & 1 & 2 & 3 & 4 & 5\\
				\hline
				2 & $42^p$ &$651^b$ &         &        &       \\
				3 & $18^{p} $ & $180^{j,g}$ & $1395^b$ &        &       \\
				4 & $6^{j,l}$  &  $21^{l,\star}$ & $121^t-126^q$  & $651^b$ &       \\
				5 & $2^i$  &  $2^i$ & $2^i$   & $32^c$  & $63^b$ \\
				\hline
			\end{tabular}
			\caption{Bounds for $\cA_2(6,k,t;2)$}
		\end{center}
	\end{table}
	
	\begin{table}[!htp]
		\begin{center}
			\begin{tabular}{|r|rrrrrr|}
				\hline
				k/t & 1 & 2 & 3 & 4 & 5 & 6\\
				\hline
				2 & $84^l$ & $2667^b$ & & & & \\
				3 & $34^{l,j}$ & $741^t-762^{j}$ & $2667^b$ & & & \\
				4 & $16^{l,j}$ & $96^t-144^{l}$ & $906^t-1524^{j}$ & $11811^b$ & & \\
				5 & $2^i$ & $7^l$ & $43^t-85^{j}$ & $360^t-478^{q}$ & $2667^b$ & \\
				6 & $2^i$ & $2^i$ & $2^i$ & $2^i$ & $64^c$ & $127^b$ \\
				\hline
			\end{tabular}
			\caption{Bounds for $\cA_2(7,k,t;2)$}
		\end{center}
	\end{table}
	

\begin{table}[!htp]
		\begin{center}
			\begin{tabular}{|r|rrrr|}
				\hline
				k/t & 1 & 2 & 3 & 4 \\
				\hline
				2 & $170^p$ & $10795^b$ & &  \\
				3 & $72^{t,j}$ & $2663^t-3060^{j}$ & $97155^b$ & \\
				4 & $34^p$ & $512^x-578^{j}$ & $6933^t-12954^{j}$ & $200787^b$ \\
				5 & $10^{t,i}$ & $33^l-128^{j}$ & $318^t-1184^{j}$ & $4821^t-12532^{j}$ \\
				6 & $2^i$ & $2^i$ & $17^t-25^{j}$ & $71^t-341^{j}$ \\
				7 & $2^i$ & $2^i$ & $2^i$ & $2^i$ \\
				\hline
				k/t & 5 & 6 & 7 & \\
				\hline
				5 & $97155^b$ & & & \\
				6 & $969^x-1870^{q}$ & $10795^b$ & & \\
				7 & $2^i$ & $128^c$ & $255^b$ & \\
				\hline
			\end{tabular}
			\caption{Bounds for $\cA_2(8,k,t;2)$}
		\end{center}
	\end{table}

	\begin{table}[!htp]
		\begin{center}
			\begin{tabular}{|r|rrr|}
				\hline
				k/t & 1 & 2 & 3\\
				\hline
				2 & $7^b$ & $7^b$ &       \\
				3 & $1^b$ & $1^b$ & $1^b$ \\
				\hline
			\end{tabular}
			\caption{Bounds for $\cA_2(3,k,t;3)$}
		\end{center}
	\end{table}
	
	\begin{table}[!htp]
		\begin{center}
			\begin{tabular}{|r|rrrr|}
				\hline
				k/t & 1 & 2 & 3 & 4 \\
				\hline
				2 & $15^p$ & $35^b$ &        &       \\
				3 & $5^{a,j}$  & $15^b$ & $15^b$ &       \\
				4 & $1^b$  &  $1^b$ & $1^b$  & $1^b$ \\
				\hline
			\end{tabular}
			\caption{Bounds for $\cA_2(4,k,t;3)$}
		\end{center}
	\end{table}
	
	\begin{table}[!htp]
		\begin{center}
			\begin{tabular}{|r|rrrrr|}
				\hline
				k/t & 1 & 2 & 3 & 4 & 5\\
				\hline
				2 & $31^l$ &$155^b$ &         &        &       \\
				3 & $11^{l,j} $  & $53^t-58^l$ & $155^b$ &        &       \\
				4 & $3^i$  &  $6^l$ & $31^b$  & $31^b$ &       \\
				5 & $1^b$  &  $1^b$ & $1^b$   & $1^b$  & $1^b$ \\
				\hline
			\end{tabular}
			\caption{Bounds for $\cA_2(5,k,t;3)$}
		\end{center}
	\end{table}
	
	\begin{table}[!htp]
		\begin{center}
			\begin{tabular}{|r|rrrrr|}
				\hline
				k/t & 1 & 2 & 3 & 4 & 5\\
				\hline
				2 & $63^p$ &$651^b$ &         &        &       \\
				3 & $27^{p} $ & $279^{j,k}$ & $1395^b$ &        &       \\
				4 & $9^{l}$  &  $35^t-43^{j}$ & $195^t-242^{j}$  & $651^b$ &       \\
				5 & $3^i$  &  $3^i$ & $8^l$   & $63^b$  & $63^b$ \\
				\hline
			\end{tabular}
			\caption{Bounds for $\cA_2(6,k,t;3)$}
		\end{center}
	\end{table}
	
	\begin{table}[!htp]
		\begin{center}
			\begin{tabular}{|r|rrrrrr|}
				\hline
				k/t & 1 & 2 & 3 & 4 & 5 & 6\\
				\hline
				2 & $127^{d}$ & $2667^b$ & & & & \\
				3 & $53^{t,j}$ & $1143^{j,k}$ & $2667^b$ & & & \\
				4 & $21^l-23^{j}$ & $150^t-227^{j}$ & $1545^t-2358^{h}$ & $11811^b$ & & \\
				5 & $7^l$ & $19^l-34^{i}$ & $76^t-173^{j}$ & $675^t-990^{j}$ & $2667^b$ & \\
				6 & $3^i$ & $3^i$ & $3^i$ & $11^l$ & $127^b$ & $127^b$ \\
				\hline
			\end{tabular}
			\caption{Bounds for $\cA_2(7,k,t;3)$}
		\end{center}
	\end{table}
	

	\begin{table}[!htp]
		\begin{center}
			\begin{tabular}{|r|rrrr|}
				\hline
				k/t & 1 & 2 & 3 & 4 \\
				\hline
				2 & $255^p$ & $10795^b$ & &  \\
				3 & $107^{t,j}$ & $4293^t-4625^{j}$ & $97155^b$ &  \\
				4 & $51^p$ & $768^x-901^{j}$ & $12977^t-19431^{j}$ & $200787^b$  \\
				5 & $18^t-21^{i}$ & $59^l-187^{j}$ & $676^t-1865^{j}$ & $9563^t-19403^{j}$  \\
				6 & $5^l$ & $15^t-22^{i}$ & $39^t-127^{i}$ & $179^t-697^{j}$ \\
				7 & $3^{i}$ & $3^{i}$ & $3^{i}$ & $3^{i}$ \\
				\hline
				k/t & 5 & 6 & 7 &\\
				\hline
	         5 & $97155^b$ & & & \\
				6 & $2341^x-4004^{j}$ & $10795^b$ & & \\
				7 & $17^l-65^l$ & $255^b$ & $255^b$ & \\
				\hline
			\end{tabular}
			\caption{Bounds for $\cA_2(8,k,t;3)$}
		\end{center}
	\end{table}

	\begin{table}[!htp]
		\begin{center}
			\begin{tabular}{|r|rrr|}
				\hline
				k/t & 1 & 2 & 3\\
				\hline
				2 & $7^b$ & $7^b$ &       \\
				3 & $1^b$ & $1^b$ & $1^b$ \\
				\hline
			\end{tabular}
			\caption{Bounds for $\cA_2(3,k,t;4)$}
		\end{center}
	\end{table}
	
	\begin{table}[!htp]
		\begin{center}
			\begin{tabular}{|r|rrrr|}
				\hline
				k/t & 1 & 2 & 3 & 4 \\
				\hline
				2 & $20^p$ & $35^b$ &        &       \\
				3 & $8^{a,j}$  & $15^b$ & $15^b$ &       \\
				4 & $1^b$  &  $1^b$ & $1^b$  & $1^b$ \\
				\hline
			\end{tabular}
			\caption{Bounds for $\cA_2(4,k,t;4)$}
		\end{center}
	\end{table}
	
	\begin{table}[!htp]
		\begin{center}
			\begin{tabular}{|r|rrrrr|}
				\hline
				k/t & 1 & 2 & 3 & 4 & 5\\
				\hline
				2 & $40^l$ &$155^b$ &         &        &       \\
				3 & $16^{j,l} $  & $80^l-82^l$ & $155^b$ &        &       \\
				4 & $6^{l,a}$  &  $16^l$ & $31^b$  & $31^b$ &       \\
				5 & $1^b$  &  $1^b$ & $1^b$   & $1^b$  & $1^b$ \\
				\hline
			\end{tabular}
			\caption{Bounds for $\cA_2(5,k,t;4)$}
		\end{center}
	\end{table}
	
	\begin{table}[!htp]
		\begin{center}
			\begin{tabular}{|r|rrrrr|}
				\hline
				k/t & 1 & 2 & 3 & 4 & 5\\
				\hline
				2 & $84^p$ &$651^b$ &         &        &       \\
				3 & $36^{p} $ & $360^{g,j}$ & $1395^b$ &        &       \\
				4 & $16^{l,j}$  &  $52^t-64^{j}$ & $336^t-342^{j}$  & $651^b$ &       \\
				5 & $4^i$  &  $7^l$ & $32^l$   & $63^b$  & $63^b$ \\
				\hline
			\end{tabular}
			\caption{Bounds for $\cA_2(6,k,t;4)$}
		\end{center}
	\end{table}
	
	\begin{table}[!htp]
		\begin{center}
			\begin{tabular}{|r|rrrrrr|}
				\hline
				k/t & 1 & 2 & 3 & 4 & 5 & 6\\
				\hline
				2 & $168^{d}$ & $2667^b$ & & & & \\
				3 & $68^l-72^{j}$ & $1524^{j,k}$ & $2667^b$ & & & \\
				4 & $30^l-32^{j}$ & $257^l-304^{j}$ & $2298^t-3048^{j}$ & $11811^b$ & & \\
				5 & $12^l$ & $33^l-64^{j}$ & $135^t-260^{j}$ & $1344^t-1398^{j}$ & $2667^b$ & \\
				6 & $4^i$ & $4^i$ & $9^l$ & $64^l$ & $127^b$ & $127^b$ \\
				\hline
			\end{tabular}
			\caption{Bounds for $\cA_2(7,k,t;4)$}
		\end{center}
	\end{table}
	

	\begin{table}[!htp]
		\begin{center}
			\begin{tabular}{|r|rrrr|}
				\hline
				k/t & 1 & 2 & 3 & 4\\
				\hline
				2 & $340^p$ & $10795^b$ & &  \\
				3 & $144^{t,j}$ & $5751^t-6120^{j}$ & $97155^b$ &  \\
				4 & $68^p$ & $1024^x-1224^{j}$ & $16963^t-25908^{j}$ & $200787^b$  \\
				5 & $27^t-31^{i}$ & $85^l-260^{j}$ & $1076^t-2498^{j}$ & $14919^t-25070^{j}$ \\
				6 & $10^t-12^{j}$ & $25^t-44^{j}$ & $71^t-256^{j}$ & $371^t-1050^{j}$  \\
				7 & $4^i$ & $4^i$ & $4^i$ & $12^l-40^l$  \\
				\hline
				k/t & 5 & 6 & 7 & \\
	         \hline
	         5 & $97155^b$ & & & \\
				6 & $5377^x-5654^{j}$ & $10795^b$ & & \\
				7 & $128^l$ & $255^b$ & $255^b$ & \\
				\hline
			\end{tabular}
			\caption{Bounds for $\cA_2(8,k,t;4)$}
		\end{center}
	\end{table}

	\begin{table}[!htp]
		\begin{center}
			\begin{tabular}{|r|rrr|}
				\hline
				k/t & 1 & 2 & 3\\
				\hline
				2 & $4^a$ & $13^b$ &       \\
				3 & $1^b$ & $1^b$ & $1^b$ \\
				\hline
			\end{tabular}
			\caption{Bounds for $\cA_3(3,k,t;2)$}
		\end{center}
	\end{table}
	
	\begin{table}[!htp]
		\begin{center}
			\begin{tabular}{|r|rrrr|}
				\hline
				k/t & 1 & 2 & 3 & 4 \\
				\hline
				2 & $20^p$ & $130^b$ &        &       \\
				3 & $2^i$  &  $10^l$ & $40^b$ &       \\
				4 & $1^b$  &  $1^b$ & $1^b$  & $1^b$ \\
				\hline
			\end{tabular}
			\caption{Bounds for $\cA_3(4,k,t;2)$}
		\end{center}
	\end{table}
	
	\begin{table}[!htp]
		\begin{center}
			\begin{tabular}{|r|rrrrr|}
				\hline
				k/t & 1 & 2 & 3 & 4 & 5\\
				\hline
				2 & $58^l-59^{j}$ &$1210^b$ &         &        &       \\
				3 & $12^l-14^l$ & $88^l-176^l$ & $1210^b$ &        &       \\
				4 & $2^i$  &  $2^i$ & $20^l$  & $121^b$ &       \\
				5 & $1^b$  &  $1^b$ & $1^b$   & $1^b$  & $1^b$ \\
				\hline
			\end{tabular}
			\caption{Bounds for $\cA_3(5,k,t;2)$}
		\end{center}
	\end{table}
	
	\begin{table}[!htp]
		\begin{center}
			\begin{tabular}{|r|rrr|}
				\hline
				k/t & 1 & 2 & 3\\
				\hline
				2 & $9^a$ & $13^b$ &       \\
				3 & $1^b$ & $1^b$ & $1^b$ \\
				\hline
			\end{tabular}
			\caption{Bounds for $\cA_3(3,k,t;3)$}
		\end{center}
	\end{table}
	
	\begin{table}[!htp]
		\begin{center}
			\begin{tabular}{|r|rrrr|}
				\hline
				k/t & 1 & 2 & 3 & 4 \\
				\hline
				2 & $30^p$ & $130^b$ &        &       \\
				3 & $5^l$  &  $27^l$ & $40^b$ &       \\
				4 & $1^b$  &  $1^b$ & $1^b$  & $1^b$ \\
				\hline
			\end{tabular}
			\caption{Bounds for $\cA_3(4,k,t;3)$}
		\end{center}
	\end{table}
	
	\begin{table}[!htp]
		\begin{center}
			\begin{tabular}{|r|rrrrr|}
				\hline
				k/t & 1 & 2 & 3 & 4 & 5\\
				\hline
				2 & $90^l$ &$1210^b$ &         &        &       \\
				3 & $27^l$ & $157^l-270^l$ & $1210^b$ &        &       \\
				4 & $3^i$  &  $11^l$ & $81^l$  & $121^b$ &       \\
				5 & $1^b$  &  $1^b$ & $1^b$   & $1^b$  & $1^b$ \\
				\hline
			\end{tabular}
			\caption{Bounds for $\cA_3(5,k,t;3)$}
		\end{center}
	\end{table}
	
	\begin{table}[!htp]
		\begin{center}
			\begin{tabular}{|r|rrr|}
				\hline
				k/t & 1 & 2 & 3\\
				\hline
				2 & $13^b$ & $13^b$ &       \\
				3 & $1^b$ & $1^b$ & $1^b$ \\
				\hline
			\end{tabular}
			\caption{Bounds for $\cA_3(3,k,t;4)$}
		\end{center}
	\end{table}
	
	\begin{table}[!htp]
		\begin{center}
			\begin{tabular}{|r|rrrr|}
				\hline
				k/t & 1 & 2 & 3 & 4 \\
				\hline
				2 & $40^p$ & $130^b$ &        &       \\
				3 & $10^l$  &  $40^b$ & $40^b$ &       \\
				4 & $1^b$  &  $1^b$ & $1^b$  & $1^b$ \\
				\hline
			\end{tabular}
			\caption{Bounds for $\cA_3(4,k,t;4)$}
		\end{center}
	\end{table}
	
	\begin{table}[!htp]
		\begin{center}
			\begin{tabular}{|r|rrrrr|}
				\hline
				k/t & 1 & 2 & 3 & 4 & 5\\
				\hline
				2 & $121^l$ &$1210^b$ &         &        &       \\
				3 & $33^l-34^{j}$ & $234^l-364^l$ & $1210^b$ &        &       \\
				4 & $6^l$  &  $20^l$ & $121^b$  & $121^b$ &       \\
				5 & $1^b$  &  $1^b$ & $1^b$   & $1^b$  & $1^b$ \\
				\hline
			\end{tabular}
			\caption{Bounds for $\cA_3(5,k,t;4)$}
		\end{center}
	\end{table}

\end{document}